\providecommand{\UseBiblatex}{0}
\ifnum\UseBiblatex=1
\documentclass[pdflatex]{sn-jnl}% Vega/biblatex preprint build; see wrapper file.
\else
\documentclass[pdflatex,sn-mathphys-num]{sn-jnl}% Math and Physical Sciences Numbered Reference Style
\fi
\usepackage[english,preprint,digital,subtle]{vega-snjnl}
\AtBeginDocument{\newgeometry{left=20mm,right=20mm,top=22mm,bottom=25mm}}

\usepackage{graphicx}%
\usepackage{multirow}%
\usepackage{amsmath,amssymb,amsfonts}%
\usepackage{amsthm}%
\usepackage{mathrsfs}%
\usepackage[title]{appendix}%
\usepackage{xcolor}%
\usepackage{textcomp}%
\usepackage{manyfoot}%
\usepackage{booktabs}%
\usepackage{algorithm}%
\usepackage{algorithmicx}%
\usepackage{algpseudocode}%
\usepackage{listings}%
\ifnum\UseBiblatex=1
\usepackage{csquotes}
\usepackage[backend=biber,style=numeric-comp,sorting=none,giveninits=true,maxbibnames=99]{biblatex}
\fi
\makeatletter
\newtheoremstyle{thmstyleone}%
  {18pt plus2pt minus1pt}{18pt plus2pt minus1pt}%
  {\itshape}{0pt}{\bfseries}{}%
  {.5em}{\thmname{\Vega@ThmLabel{#1}}%
         \thmnumber{\@ifnotempty{#1}{ \Vega@ThmLabel{\@upn{#2}}}}%
         \thmnote{ {\the\thm@notefont(#3)}}.}
\newtheoremstyle{thmstyletwo}%
  {18pt plus2pt minus1pt}{18pt plus2pt minus1pt}%
  {\normalfont}{0pt}{\itshape}{}%
  {.5em}{\thmname{\Vega@ThmLabel{#1}}%
         \thmnumber{\@ifnotempty{#1}{ \Vega@ThmLabel{#2}}}%
         \thmnote{ {\the\thm@notefont(#3)}}.}
\newtheoremstyle{thmstylethree}%
  {18pt plus2pt minus1pt}{18pt plus2pt minus1pt}%
  {\normalfont}{0pt}{\bfseries}{}%
  {.5em}{\thmname{\Vega@ThmLabel{#1}}%
         \thmnumber{\@ifnotempty{#1}{ \Vega@ThmLabel{\@upn{#2}}}}%
         \thmnote{ {\the\thm@notefont(#3)}}.}
\makeatother

\theoremstyle{thmstyleone}%
\newtheorem{theorem}{Theorem}%  meant for continuous numbers
\newtheorem{proposition}[theorem]{Proposition}% 

\theoremstyle{thmstyletwo}%
\newtheorem{remark}{Remark}%

\theoremstyle{thmstylethree}%
\newtheorem{assumption}[theorem]{Assumption}
\newtheorem{lemma}[theorem]{Lemma}
\newtheorem{corollary}[theorem]{Corollary}

\newcommand\F{\mbox{I\kern-2pt F}}

\newcommand\cC{{\cal C}}
\newcommand\cF{{\cal F}}

\newcommand\cL{{\cal L}}

\newcommand\cN{{\cal N}}

\newcommand\cX{{\cal X}}

\newcommand\cP{{\cal P}}

\def\R{{\mathbb R}}
\def\N{{\mathbb N}}

\newcommand\E{{\mathbb{E}}}

\newcommand\DD{\mathrm D}

\newcommand\beq{\begin{equation}}
\newcommand\eeq{\end{equation}}
\newcommand\bea{\begin{eqnarray}}
\newcommand\eea{\end{eqnarray}}
\newcommand\bean{\begin{eqnarray*}}
\newcommand\eean{\end{eqnarray*}}
\begin{document}

\title[Convergence rates]{Convergence Rates of Continuous-Time Random Walks to Time-Fractional Diffusions with Unbounded Coefficients}


\author*[1,2,3]{\fnm{Artur} \sur{Sidorenko}}\email{sidorenkoap@my.msu.ru}
\author*[1,2,3]{\fnm{Vasilii} \sur{Kolokoltsov}}\email{kolokoltsov59@mail.ru}

\affil[1]{\orgname{Lomonosov Moscow State University}, \orgaddress{ \city{Moscow}, \country{Russia}}}

\affil[2]{\orgname{National Research University Higher School of Economics}, \orgaddress{ \city{Moscow}, \country{Russia}}}

\affil[3]{\orgname{Vega Institute Foundation}, \orgaddress{ \city{Moscow}, \country{Russia}}}

%%==================================%%
%% Abstract %%
%%==================================%%

\abstract{We investigate uniform weak convergence rates for probabilistic numerical
methods applied to backward time-fractional diffusion equations whose dynamics
are driven by diffusions with possibly unbounded coefficients, such as the
Geometric Brownian Motion.  The fractional structure is represented through a
random time-change by the inverse of a stable subordinator. To approximate the underlying fractional
dynamics, we combine discrete Markov chain schemes for the diffusion component
with heavy-tailed random walk approximations of the time change.

Our analysis builds on Feller semigroup techniques and a high-order
sensitivity framework for diffusion semigroups based on the Kunita stochastic
flows and tensor fields.  We derive uniform bounds for all orders of
sensitivities, establish a quasi-contraction property for the associated
semigroup, and transfer these estimates to the fractional setting via the
convolution representation with the inverse subordinator.  As a result, under
killing conditions which dominate at least the base-space semigroup growth, we
obtain weak convergence rates for the combined continuous-time-random-walk
scheme to the time-fractional diffusion, with a logarithmic regime before the
discount dominates the stronger smooth-space growth.}



\maketitle

\tableofcontents

\section{Introduction}

We consider the following time-fractional Cauchy problem
\begin{equation}
    \label{eq:cauchy}
    \partial_t^\beta u(t,x)
    = \frac{1}{\Gamma(1-\beta)}\bigl(Lu(t,x) - c\,u(t,x)\bigr),
    \qquad u(0,x) = f(x),
\end{equation}
where $c \ge 0$, and
\begin{equation}
    \partial_t^\beta g(t)
    = D_{0+ \star}^{\beta} g(t)
    := \frac{1}{\Gamma(1-\beta)} \int_0^t (t-s)^{-\beta} g'(s)\,ds
\end{equation}
is the Caputo--Dzherbashian fractional derivative acting in the time variable $t$, whereas $L$ is the generator of a diffusion process acting in the space variable $x$.  

The solution to \eqref{eq:cauchy} admits the following probabilistic representation (see, e.g.\ the monographs \cite{Kolokoltsov2011Markov,Kolokoltsov2010Nonlinear})
\begin{equation}
    \label{eq:probsol}
    u(t,x) = \E_{x}\bigl[e^{-c \hat\sigma_t} f(X_{\hat\sigma_t})\bigr],
\end{equation}
where $X$ is a diffusion process with generator $L$, and $\hat\sigma_t$ is the inverse of a $\beta$-stable subordinator.  
The representation \eqref{eq:probsol} leads to numerical schemes based on continuous-time random walks (CTRWs).  
The main question of interest is to obtain quantitative convergence rates of these CTRW-based approximations to the true solution $u$.  The equation itself is meaningful for $c\ge0$; the fractional convergence theorem proved below is stated in the killed regimes where the discount parameter $c$ dominates at least the base-space growth rate, with a logarithmic rate before it dominates the stronger smooth-space growth rate.

For bounded diffusion coefficients, weak convergence and convergence rates of CTRW approximations were obtained in \cite{Kolokoltsov2023Rates}.  
In this paper we extend these results to the case of unbounded diffusion coefficients.  
This setting includes affine coefficients.  These examples are covered in
Theorem~\ref{theo:frac-abstract} and Corollary~\ref{cor:frac-weighted-regular}.

We work with continuous-time random walk (CTRW) approximations.  
CTRW processes, originally introduced in physics \cite{montroll1965random}, are known to have scaling limits described by Markov processes with a random time change given by inverse stable subordinators \cite{kolokoltsov2009generalized}.  
Numerical methods for time-fractional PDEs have been developed both in the deterministic and probabilistic settings; here we use the probabilistic CTRW approximation framework of \cite{Kolokoltsov2023Rates}.

Our analysis relies on the Feller semigroup techniques of \cite{Kolokoltsov2011Markov,Kolokoltsov2010Nonlinear} and on CTRW approximations of stable subordinators proposed in \cite{Kolokoltsov2023Rates}.
To obtain convergence rates for our CTRW approximations, we analyse high-order sensitivities of the diffusion semigroup,
\[
    \DD^k F_t f(x),
\]
and establish a quasi-contraction property for the diffusion Feller semigroup $F_t$ on spaces of smooth functions with bounded (weighted) derivatives.  
We use Kunita’s stochastic flow theory \cite{Kunita1984SaintFlour,Kunita1990,Kunita2019}, which allows for pathwise analysis of derivatives.  
Under mild conditions higher-order derivatives can be organised into tensor fields.  
We employ a recent chain rule for tensor fields \cite{Licht2024}, originating from the Fa\`a di Bruno formula \cite{Comtet2012,Savits2006}.

Our main contributions are as follows:
\begin{enumerate}
    \item We obtain convergence rate estimates for CTRW schemes approximating the probabilistic solution \eqref{eq:probsol} of the killed time-fractional Cauchy problem \eqref{eq:cauchy} with unbounded diffusion coefficients, including a logarithmic fractional rate when the killing dominates the base-space growth but not the stronger smooth-space growth.
    \item We prove uniform convergence rates for random-walk approximations of diffusions with unbounded coefficients, at the level of Feller semigroups.
    \item We develop a methodology for investigating sensitivities of diffusion processes
    \[
        \DD^m X_t(x,\omega)
    \]
    with respect to the initial condition $x$ as random tensor fields, based on Kunita’s stochastic flow theory and the tensorial chain rule, and derive corresponding growth estimates.
\end{enumerate}

The structure of the paper is as follows.  
Section~\ref{sec:main} contains the precise setting and the formulation of the main results.  
Section~\ref{sec:tensor} introduces the stochastic tensor field notation, the tensorial chain rule, and the basic elements of Kunita’s theory used in this work.  
In Section~\ref{sec:sensitivities} we study the sensitivities of the diffusion semigroup with respect to the initial point and establish the quasi-contraction property.  
Section~\ref{sec:weighted-spaces} records the weighted Feller facts used for unbounded observables.
Section~\ref{sec:proof} contains the proofs of the main results, including the weak convergence rates for the time-fractional diffusion.

\subsection{Notation}

We collect here the conventions used in the statements of the main results.
Throughout, $\R_+=[0,\infty)$ and $\N_0=\N\cup\{0\}$.  If $r\in\R$,
then $[r]$ denotes its integer part.  The state space is
$\cX=\R^k$, equipped with the Euclidean norm $|\cdot|$.  Constants denoted
by $C$, possibly with subscripts, are positive and may change from line to
line; their dependence is recorded when it is relevant for an estimate.

For a diffusion started from $x$, expectations are denoted by $\E_x$.  The
diffusion semigroup is
\[
    F_t f(x):=\E_x f(X_t),
\]
and $L$ denotes its generator.  The killed semigroup is written
$e^{-ct}F_t$.  The one-step Markov approximation with time step $h$ is
denoted by $U_h$; thus $U_h^n$ is the $n$-step transition operator, and in
continuous-time estimates we use the shorthand
$U_h^{\lfloor s/h\rfloor}$ for the approximation at operational time $s$.

The fractional clock is denoted as follows.  The process
$\hat\Sigma$ is a $\beta$-stable subordinator and
$\hat\sigma_T:=\sup\{s:\hat\Sigma_s\le T\}$ is its inverse.  Its discrete
counterpart is generated by the heavy-tailed walk
$S_n^h=h^{1/\beta}\sum_{i=1}^n\tau_i$; the first passage index and the
corresponding operational time are
\[
    n_T^h:=\min\{n\ge0:S_n^h>T\},
    \qquad
    N_T^h:=h\,n_T^h.
\]

Derivatives with respect to the initial point are written in tensor form:
$\DD^j f(x)$ is the $j$-linear derivative of $f$ at $x$, with the convention
$\DD^0f=f$.  Tensor arguments are placed in square brackets, for example
$\DD^j f(x)[v_1,\dots,v_j]$, and their norms are the operator tensor norms
introduced in \eqref{eq:tensor_norm}.  Symmetric tensor products are denoted
by $\vee$.

We use two regularity scales.  For a weight $G\ge1$, $C_G(\cX)$ and
$C_{G,\infty}(\cX)$ are the weighted continuous spaces with norm
$\|f\|_G=\sup_x |f(x)|/G(x)$.  The bounded-derivative scale uses
\[
    V(x)=1+|x|^2,\qquad W(x)=1+|x|^4,\qquad D_l\subset C_{W,\infty}(\cX),
\]
with norm $\|\cdot\|_l$ defined below from bounded derivatives and one base
point value.  The weighted-regular scale uses
$w_\alpha(x)=(1+|x|^2)^{\alpha/2}$ and the spaces
$\cC^r_{\alpha,\infty}(\cX)$ with norm $\|\cdot\|_{\alpha,r}$.  In the
abstract fractional theorem, $B$ denotes the base Banach space in which the
error is measured, $D\subset B$ denotes the smoother space controlling the
local approximation error, $\bar\mu_B$ is the base-space growth rate, and
$\mu_D$ is the corresponding growth rate in the smoother norm.

\section{Setting and Main Results} \label{sec:main}

Let $(\Omega,\cF,(\cF_t)_{t\in\R_+},P,W)$ be a stochastic basis such that
\begin{enumerate}
    \item $\cF$ is complete;
    \item $\cF_0$ contains all $P$-null sets $\cN$ from $\cF$;
    \item $W_t=(W_t^1,\dots,W_t^d)$ is a standard $d$-dimensional Brownian motion;
    \item $\cF_t=\sigma\{\cN,W_s:0\le s\le t\}$.
\end{enumerate}
Thus the filtration satisfies the usual conditions.  We also fix, independently of $W$, a $\beta$-stable subordinator $\hat\Sigma$ with generator
\[
    \hat L_\beta g(t) := \int_t^\infty \frac{g(s) - g(t)}{(s-t)^{1 + \beta}} \, ds.
\]

Put $\cX:=\R^k$, let $\sigma_j:\cX\to\cX$ for $j=0,\dots,d$, and write $b:=\sigma_0$.  We consider the It\^o SDE
\begin{equation}
    \label{eq:SDE-flow}
    dX_t = b(X_t)\,dt + \sum_{j=1}^d \sigma_j(X_t)\,dW_t^j, \qquad X_0 = x \in \cX.
\end{equation}

For $m\in\N$ and $\theta\in(0,1)$, let $C^{m,\theta}_g(\cX;\cX)$ denote the maps whose derivatives up to order $m$ are bounded and continuous and whose $m$-th derivative is globally $\theta$-H\"older continuous.
\begin{assumption} \label{asm:classical}
    The coefficients of \eqref{eq:SDE-flow} satisfy the following:
    \begin{enumerate}
        \item for all $x\in\cX$,
        \[
            |\sigma_j(x)| \leq K (1 + |x|), \qquad j = 0, \dots, d,
        \]
        for some $K>0$;
        \item for all $x,x'\in\cX$,
        \[
            |\sigma_j(x) - \sigma_j(x')| \leq K |x - x'|, \qquad j = 0, \dots, d;
        \]
        \item for some integer $m\ge4$ and some $\theta\in(0,1)$, the coefficients
        $\sigma_j$ belong to $C^{m,\theta}_g(\cX;\cX)$ and
        \[
            \|\DD^i \sigma_j(x)\| \le K
            \quad\text{for all } x\in\cX,\ i=1,\dots,m,\ j=0,\dots,d,
        \]
        in the tensor norm \eqref{eq:tensor_norm}.
    \end{enumerate}
\end{assumption}

Under Assumption~\ref{asm:classical}, \eqref{eq:SDE-flow} has a unique strong solution $X_t(x)$ for every $x\in\cX$, and $X_t(x)\in L^p$ for all $p\ge2$.

We introduce weight functions
\[
    V(x) := 1 + |x|^2, \qquad W(x) := 1 + |x|^4.
\]
We use two related scales.  First, for a weight $G\ge1$, let $C_G(\cX)$ be the space of continuous functions with
\[
    \|f\|_G := \sup_{x \in \cX} \frac{|f(x)|}{G(x)} < \infty.
\]
The subspace $C_{G,\infty}(\cX)$ consists of those $f\in C_G(\cX)$ for which $f/G$ vanishes at infinity.
Fix a base point $x_*\in\cX$; for the ambient state space $\R^k$ we take
$x_*=0$.
For $1 \le l \le m$, we introduce the subspace
\[
    D_l := \Bigl\{
        f \in C^l(\cX) :
        f/W\in C_\infty(\cX)\text{ and }
        \|\DD^j f(\cdot)\| \in C_\infty(\cX)
        \text{ for all } j=1,\dots,l
    \Bigr\}.
\]
The norm on $D_l$ is
\[
    \| f \|_{l} := |f(x_*)| + \sum_{j=1}^l \sup_{x \in \cX} \| \DD^j f (x) \|,
\]
where derivative norms are the tensor norms from \eqref{eq:tensor_norm}.
Since $f\in D_l$ has bounded first derivative, $|f(x)|\le C_f(1+|x|)$ and
therefore $D_l\subset C_{W,\infty}(\cX)$ continuously.  Thus the
bounded-derivative theorem below gives a coarse $W$-norm estimate for
bounded-derivative observables; the weighted-regular theorem is the sharper
scale for genuinely weighted observables.  The vanishing-at-infinity condition
on each derivative is the natural analogue of the
$\cC^r_{\alpha,\infty}$-condition used in the weighted-regular scale; it is
what makes the semigroup $(F_t)_{t\ge0}$ strongly continuous on $D_l$ in the
norm $\|\cdot\|_l$ under linear-growth coefficients (Lemma~\ref{lem:D4-strong-continuity-core}
below).

Second, for the weighted regular estimates put
\[
    w_\alpha(x):=(1+|x|^2)^{\alpha/2}.
\]
For $k\in\N_0$ define
\[
    \|f\|_{\alpha,k}
    :=
    \max_{0\le j\le k}
    \sup_{x\in\cX}
    \frac{\|\DD^j f(x)\|}{w_{\alpha-j}(x)}.
\]
Let $\cC^k_\alpha(\cX)$ be the space of all $f\in C^k(\cX)$
with $\|f\|_{\alpha,k}<\infty$, and by $\cC^k_{\alpha,\infty}(\cX)$ the
closed subspace of functions satisfying
\[
    \frac{\|\DD^j f(\cdot)\|}{w_{\alpha-j}(\cdot)}\in C_\infty(\cX),
    \qquad 0\le j\le k.
\]
These spaces carry the norm $\|\cdot\|_{\alpha,k}$.

\begin{assumption}\label{asm:weighted-regular}
    For fixed $\alpha\in\R$ and $r\in\N_0$, the coefficients satisfy
    \[
        \sigma_j\in \cC^{r+4}_1(\cX;\cX),
        \qquad j=0,\dots,d,
    \]
    and their $\cC^{r+4}_1$-norms are bounded by a common constant
    $K_{\alpha,r}$.  Thus, for every $q=0,\dots,r+4$,
    \[
        \|\DD^q\sigma_j(x)\|
        \le K_{\alpha,r} w_{1-q}(x),
        \qquad x\in\cX,\quad j=0,\dots,d.
    \]
    In addition, $\DD^{r+4}\sigma_j$ is globally $\theta$-H\"older continuous for some $\theta\in(0,1)$.
\end{assumption}
\begin{remark}
    Affine coefficients, including Black--Scholes and Langevin-type coefficients, satisfy this condition.
\end{remark}

The one-step random-walk approximation of \eqref{eq:SDE-flow} is given by the transition operator $U_h$.
\begin{assumption} \label{asm:rw}
    The random--walk transition operator is
    \[
        (U_h f)(x) := \E\Bigl[ f\bigl(x + b(x)h + \sum_{j=1}^d \sigma_j(x)\,\xi_j\sqrt{h}\bigr) \Bigr],
    \]
    where $\xi$ is an $\R^d$-valued random vector with
\[
    \E[\xi]=0,\qquad \E[\xi\wedge\xi]=I_d,\qquad M_4 := \E\|\xi\|^4 < \infty,
\]
and vanishing third moments:
\[
    \E[\xi_{i_1}\xi_{i_2}\xi_{i_3}] = 0, \qquad i_1,i_2,i_3 = 1,\dots,d.
\]
\end{assumption}
\begin{remark}
    The random walk replaces the Brownian increment.  The choice $\xi=h^{-1/2}W_h$ gives the Euler--Maruyama scheme.  The odd moments vanish, for example, when $\cL(\xi)=\cL(-\xi)$.
\end{remark}

\begin{assumption} \label{asm:rw-bounded}
    The random vector $\xi$ from Assumption~\ref{asm:rw} is bounded:
    \[
        \|\xi\|\le R_\xi
        \qquad\text{a.s.}
    \]
\end{assumption}

\begin{assumption} \label{asm:rw-weighted-stability}
    For given $\alpha\in\R$ and $r\in\N_0$, there are $q_{\alpha,r}\ge0$ and $h_0>0$ such that
    \[
        \|U_h g\|_{\alpha,r}
        \le e^{q_{\alpha,r}h}\|g\|_{\alpha,r},
        \qquad
        g\in\cC^r_{\alpha,\infty}(\cX),\quad 0<h\le h_0.
    \]
\end{assumption}
\begin{remark}
    For $r=0$ this follows from the usual weighted Lyapunov estimate for
    $x+\eta_h(x)$, exactly as in Lemma~\ref{lem:U-Lyapunov-W} for the weight
    $W$.  For $r\ge1$ it is a genuinely stronger, derivative-level stability
    condition: it is the one-step discrete analogue of the weighted jet bounds
    in Lemma~\ref{lem:weighted-regular-semigroup-core}.  We impose it as a
    hypothesis rather than derive it, since its verification is
    scheme-dependent; establishing it for concrete walks under
    Assumption~\ref{asm:weighted-regular} is left open.
\end{remark}

The first convergence result is the bounded-derivative version of the $C_W$ argument.

\begin{theorem} \label{prop:uniform}
    Let $F_t$ be the Feller semigroup of \eqref{eq:SDE-flow}, and let
    Assumptions~\ref{asm:classical} and~\ref{asm:rw} hold.  Then, for any
    $f\in D_4$,
    \begin{equation}
        \sup_{s \leq t} \| U_h^{\lfloor s/h\rfloor} f - F_s f \|_W \leq  C_{\mathrm{RW}} e^{\mu t} h \|f \|_4,
    \end{equation}
    where $\mu:=\mu_4$ is the renormed $D_4$-growth rate of
    Corollary~\ref{coro:norm_estim}, and $C_{\mathrm{RW}}$ depends only on the constants in Assumptions~\ref{asm:classical} and~\ref{asm:rw}.  For the killed semigroups,
    \begin{equation}
        \sup_{s \leq t} \| e^{-c\lfloor s/h\rfloor h} U_h^{\lfloor s/h\rfloor} f - e^{-c s} F_s f \|_W \leq C_{\mathrm{RW}}' e^{(\mu - c) t} h \|f \|_4,
    \end{equation}
    with $C_{\mathrm{RW}}'$ depending on the same data and on $c$.
\end{theorem}

The companion estimate uses the weighted regular scale $\cC^r_{\alpha,\infty}(\cX)$ and the bounded-increment assumption.

\begin{theorem} \label{theo:uniform-weighted-regular}
    Fix $\alpha\in\R$ and $r\in\N_0$. Let Assumptions
    \ref{asm:classical}, \ref{asm:weighted-regular}, \ref{asm:rw}, and
    \ref{asm:rw-bounded} hold, and let the walk satisfy
    Assumption~\ref{asm:rw-weighted-stability}.  Suppose also that
    Assumption~\ref{asm:classical} is valid for an order $m\ge r+5$.
    Let $F_t$ be the semigroup of \eqref{eq:SDE-flow}. Then, for every
    $f\in \cC^{r+4}_{\alpha,\infty}(\cX)$,
    \begin{equation} \label{eq:uniform-weighted-regular}
        \sup_{s\le t}
        \|U_h^{\lfloor s/h\rfloor}f-F_s f\|_{\alpha,r}
        \le
        C_{\alpha,r} e^{\mu_{\alpha,r}t}h
        \|f\|_{\alpha,r+4}.
    \end{equation}
    For the killed semigroups, if $c\ge0$, then
    \begin{equation} \label{eq:uniform-weighted-regular-killed}
        \sup_{s\le t}
        \|e^{-c\lfloor s/h\rfloor h}U_h^{\lfloor s/h\rfloor}f-e^{-cs}F_s f\|_{\alpha,r}
        \le
        C'_{\alpha,r}e^{(\mu_{\alpha,r}-c)t}h
        \|f\|_{\alpha,r+4}.
    \end{equation}
    The estimate holds for all sufficiently small $h>0$.  The constants
    depend on $\alpha,r,K_{\alpha,r},d,R_\xi$, and the moment constants in
    Assumption~\ref{asm:rw}.
\end{theorem}

The inverse stable subordinator is
\begin{equation}
    \hat\sigma_T := \sup \{s \colon \hat\Sigma_s \leq T  \}.
\end{equation}
Recall from \cite{Kolokoltsov2011Markov,Kolokoltsov2023Rates} that the Caputo--Dzherbashian derivative also has the representation
\begin{equation}
    D_{0+ \star}^{\beta} g(t) = \frac{1}{\Gamma(-\beta)} \int_0^t \frac{g(s) - g(t)}{s^{1+\beta}} ds + \frac{g(t) - g(0)}{\Gamma(1 - \beta) t^\beta} = - \frac{1}{\Gamma(1 - \beta)} \tilde L_\beta g(t),
\end{equation}
where $\tilde L_\beta$ is the generator of a decreasing $\beta$-stable subordinator started at $t$ and stopped at zero.

We approximate the stable subordinator by the heavy-tailed random walk
    \begin{equation*}
        S_n^h:=h^{1/\beta}\sum_{i=1}^n\tau_i,
        \qquad
        \Phi^h_s := S_{\lfloor s/h\rfloor}^h,
    \end{equation*}
    where $\tau_i$ are i.i.d. positive random variables.  For a terminal
    time $T>0$ we define the discrete inverse clock by
    \[
        n_T^h:=\min\{n\ge0:S_n^h>T\},
        \qquad
        N_T^h:=h\,n_T^h.
    \]
    \begin{assumption} \label{asm:heavy-tail}
        There is a probability density $p(x)$ of $\tau_i$ such that 
        \begin{equation*}
            p(x) = x^{-1-\beta}, \, x > B > 0,
        \end{equation*}
        and $p$ is otherwise arbitrary on $(0,B)$.
        Moreover, $p(0)=0$ and $p$ is continuously differentiable on $\R_+$.
    \end{assumption}
    Under this assumption, $\beta B^\beta > 1$. 
    \begin{remark}
        More general heavy-tail assumptions are possible; we use this form to keep the clock estimates explicit.
    \end{remark}

    Let $T^\beta_t$ denote the Feller semigroup of the stable subordinator $\hat\Sigma_t$.  We recall two estimates for the stable clock and its discrete approximation.
    We recall \cite[Theorem~3.1]{Kolokoltsov2023Rates}.
    \begin{theorem} \label{theo:frac1}
        Under Assumption \ref{asm:heavy-tail}, the following estimate holds
        \begin{equation}
            \int_0^\infty | P(\Phi^h_t > a) - P(\hat\Sigma_t > a)| dt \leq (1 + a + a^{-1}) C(\beta) h^{\chi(\beta)},
        \end{equation} 
        where $a>0$, and $C(\beta) > 0, \, \chi(\beta) \in (0, 1)$ can be found explicitly.
    \end{theorem}

    We also use \cite[Proposition~6]{Kolokoltsov2023Rates}.
    \begin{proposition} \label{prop:frac2}
        Let Assumption \ref{asm:heavy-tail} hold. Then, for any $a > 1$, $h < 1$ and $t > t_0$, where $t_0$ can be calculated explicitly, we have
        \begin{equation}
            P(\Phi^h_t < a) \leq 2^{1 + 2 / \beta} a t^{-1/\beta}.
        \end{equation}
    \end{proposition}

    The forward estimate of Theorem~\ref{theo:frac1} transfers to the
    inverse clocks $\hat\sigma_T$ and $N_T^h$:
    \begin{lemma}[inverse-clock rate]\label{lem:inverse-clock-rate}
        Under Assumption~\ref{asm:heavy-tail}, for every terminal time
        $T>0$ and all $h>0$,
        \begin{equation} \label{eq:inverse-clock-rate}
            \int_0^\infty
            \bigl|
                P(\hat\sigma_T>s)-P(N_T^h>s)
            \bigr|\,ds
            \le
            C_{\mathrm{clock}}(T)\,h^{\chi(\beta)},
        \end{equation}
        with $C_{\mathrm{clock}}(T):=(1+T+T^{-1})C(\beta)$, where $C(\beta)$
        and $\chi(\beta)$ are the constants of Theorem~\ref{theo:frac1}.
    \end{lemma}
    \begin{proof}
        By the duality between a non-decreasing process and its
        inverse, $\{\hat\sigma_T>s\}=\{\hat\Sigma_s\le T\}$ pathwise.  For
        the discrete clock, since
        $n_T^h=\min\{n\ge0:S_n^h>T\}$ and $N_T^h=hn_T^h$,
        \[
            \{N_T^h>s\}
            =\{n_T^h>\lfloor s/h\rfloor\}
            =\{S^h_{\lfloor s/h\rfloor}\le T\}
            =\{\Phi_s^h\le T\},
        \]
        because $S^h$ is non-decreasing. Therefore
        \[
            |P(\hat\sigma_T>s)-P(N_T^h>s)|
            =|P(\hat\Sigma_s\le T)-P(\Phi_s^h\le T)|
            =|P(\Phi_s^h>T)-P(\hat\Sigma_s>T)|.
        \]
        Integrating in $s$ and applying Theorem~\ref{theo:frac1} with $a=T$
        gives \eqref{eq:inverse-clock-rate}.
    \end{proof}

We approximate the representation \eqref{eq:probsol} by a continuous-time random walk.  In the simplest case, a CTRW is a sum
\[
    \sum_{i=1}^{N^h_t} X^h_i,
\]
where $(X_i^h)$ are i.i.d. increments.  Here the spatial motion is the discrete Markov chain with transition operator $U_h$, sampled at the inverse clock $N_t^h$.

We state the fractional estimate in an abstract form covering both spatial scales.  Let $(B,\|\cdot\|_B)$ and $(D,\|\cdot\|_D)$ be Banach spaces of functions on $\cX$, with $D\subset B$.  For $f\in D$ put
\[
    u(T,\cdot):=\E\,e^{-c\hat\sigma_T}F_{\hat\sigma_T}f,
    \qquad
    u_h(T,\cdot):=\E\,e^{-c h n_T^h}U_h^{n_T^h}f.
\]

\begin{theorem}\label{theo:frac-abstract}
    Let Assumption~\ref{asm:heavy-tail} hold.  Assume that $(F_t)_{t\ge0}$
    is a strongly continuous semigroup on $B$ with generator $L_B$,
    $D\subset\mathrm{Dom}(L_B)$, and the restriction of $L_B$ to $D$
    coincides with $L$.  Assume further that the map
    $s\mapsto \hat F_s f := e^{-cs}F_s f$ belongs to $C^1([0,\infty);B)$
    with $\tfrac{d}{ds}\hat F_s f=\hat F_s(L-c)f$.  Finally, assume that
    there exist constants $\bar\mu_B\le \mu_D$ and $C_B,C_{\mathrm{RW}},C_L>0$
    such that
    \[
        \|e^{-cs}F_s f\|_B
        +
        \|e^{-c\lfloor s/h\rfloor h}U_h^{\lfloor s/h\rfloor}f\|_B
        \le
        C_B e^{(\bar\mu_B-c)s}\|f\|_D,
        \qquad s\ge0,
    \]
    \[
        \|e^{-cs}F_s f-e^{-c\lfloor s/h\rfloor h}U_h^{\lfloor s/h\rfloor}f\|_B
        \le
        C_{\mathrm{RW}}e^{(\mu_D-c)s}h\|f\|_D,
        \qquad s\ge0,
    \]
    and
    \[
        \|e^{-cs}F_s(L-c)f\|_B
        \le C_L e^{(\bar\mu_B-c)s}\|f\|_D,
        \qquad s\ge0.
    \]
    If $c\ge\mu_D\ge\bar\mu_B$, then for $T>t_0$ and all sufficiently small
    $h$,
    \begin{equation} \label{eq:frac-abstract}
        \|u_h(T,\cdot)-u(T,\cdot)\|_B
        \le
        C_{\mathrm{frac}}\bigl(h+h^{\chi(\beta)}\bigr)\|f\|_D.
    \end{equation}
    If $\bar\mu_B\le c<\mu_D$, then for $T>t_0$ and all sufficiently small
    $h$,
    \begin{equation} \label{eq:frac-abstract-log}
        \|u_h(T,\cdot)-u(T,\cdot)\|_B
        \le
        C_{\mathrm{frac}}^{\log}
        \Bigl((\log(1/h))^{-1/\beta}+h^{\chi(\beta)}\Bigr)\|f\|_D.
    \end{equation}
    The constants depend on the clock parameters, on $T$, on
    $C_{\mathrm{clock}}(T)$, and on the displayed bounds.  In the
    logarithmic case the admissible upper bound on $h$ may also depend on
    $\mu_D-c$.
\end{theorem}

\begin{corollary}
\label{cor:frac-bounded-derivatives}
    Let Assumptions~\ref{asm:classical}, \ref{asm:rw}, and
    \ref{asm:heavy-tail} hold.  Let
    $\bar\mu_W,\mu_4\ge0$ be admissible growth bounds on
    $C_{W,\infty}(\cX)$ and $D_4$, respectively, with the equivalent-norm
    renorming used in the proof of Theorem~\ref{prop:uniform}.  If
    $c\ge\mu_4\ge\bar\mu_W$, then for every $f\in D_4$,
    \begin{equation} \label{eq:frac_main}
        \|u_h(T,\cdot)-u(T,\cdot)\|_W
        \le
        C_{\mathrm{frac}}\bigl(h+h^{\chi(\beta)}\bigr)\|f\|_4.
    \end{equation}
    If $\bar\mu_W\le c<\mu_4$, then
    \begin{equation} \label{eq:frac_main_log}
        \|u_h(T,\cdot)-u(T,\cdot)\|_W
        \le
        C_{\mathrm{frac}}^{\log}
        \Bigl((\log(1/h))^{-1/\beta}+h^{\chi(\beta)}\Bigr)\|f\|_4.
    \end{equation}
\end{corollary}

\begin{corollary}[Black--Scholes coefficients]\label{cor:black-scholes-worked}
    Consider the componentwise geometric Brownian motion
    \[
        dX_t^i=\mu_iX_t^i\,dt+\sigma_iX_t^i\,dW_t^i,
        \qquad i=1,\dots,k,
    \]
    started from $x\in(0,\infty)^k$.  The natural state space is the
    invariant domain $(0,\infty)^k$, although the coefficients extend
    linearly to $\R^k$.  Put
    \[
        K_{\mathrm{GBM}}
        :=
        \max_{1\le i\le k}\{|\mu_i|,|\sigma_i|\}.
    \]
    Then Assumption~\ref{asm:classical} holds with
    $K=K_{\mathrm{GBM}}$ and $d=k$: first derivatives are bounded by
    $K_{\mathrm{GBM}}$ and all derivatives of order at least two vanish.
    For the quartic weight $W=1+|x|^4$, Lemma~\ref{lem:weighted_estim_tensor_poly}
    gives the explicit Lyapunov coefficient
    \[
        C_2=(8K_{\mathrm{GBM}}+24K_{\mathrm{GBM}}^2)d.
    \]
    If the random-walk and clock assumptions in
    Corollary~\ref{cor:frac-bounded-derivatives} also hold for the chosen
    discretization, then the bounded-derivative fractional estimate
    \eqref{eq:frac_main} applies to $f\in D_4$ whenever
    $c\ge\mu_4\ge\bar\mu_W$, where $\mu_4$ is the renormed $D_4$ growth
    bound from Corollary~\ref{coro:norm_estim} and $\bar\mu_W$ may be taken
    from the above Lyapunov coefficient.  In the scalar case
    $\mu=0.05$, $\sigma=0.2$, $K_{\mathrm{GBM}}=0.2$, and $d=1$, this gives
    $C_2=8(0.2)+24(0.2)^2=2.56$; the minimum killing condition in this
    theorem remains $c\ge\mu_4$.
    The estimate is stated for the linear extension of the coefficients to
    $\R^k$; the one-step increment $x+\eta_h(x)$ may leave the positivity
    domain $(0,\infty)^k$, so positivity-preserving discretizations are a
    separate matter not addressed here.
\end{corollary}

\begin{corollary}
\label{cor:frac-weighted-regular}
    Fix $\alpha\in\R$ and $r\in\N_0$.  Let Assumptions
    \ref{asm:classical}, \ref{asm:weighted-regular}, \ref{asm:rw},
    \ref{asm:rw-bounded}, \ref{asm:rw-weighted-stability}, and
    \ref{asm:heavy-tail} hold.  Let
    $\bar\mu_{\alpha,r}$ and $\mu_{\alpha,r+4}$ denote admissible growth
    bounds for $F_t$ and $U_h^{\lfloor s/h\rfloor}$ on
    $\cC^r_{\alpha,\infty}(\cX)$ and for $F_t$ on
    $\cC^{r+4}_{\alpha,\infty}(\cX)$, respectively.  If
    $c\ge\mu_{\alpha,r+4}\ge\bar\mu_{\alpha,r}$, then for every
    $f\in\cC^{r+4}_{\alpha,\infty}(\cX)$,
    \begin{equation} \label{eq:frac-weighted-regular}
        \|u_h(T,\cdot)-u(T,\cdot)\|_{\alpha,r}
        \le
        C_{\alpha,r,\mathrm{frac}}\bigl(h+h^{\chi(\beta)}\bigr)
        \|f\|_{\alpha,r+4}.
    \end{equation}
    If $\bar\mu_{\alpha,r}\le c<\mu_{\alpha,r+4}$, then
    \begin{equation} \label{eq:frac-weighted-regular-log}
        \|u_h(T,\cdot)-u(T,\cdot)\|_{\alpha,r}
        \le
        C_{\alpha,r,\mathrm{frac}}^{\log}
        \Bigl((\log(1/h))^{-1/\beta}+h^{\chi(\beta)}\Bigr)
        \|f\|_{\alpha,r+4}.
    \end{equation}
\end{corollary}

\begin{remark}
    The killing coefficient is needed because the operational time
    $\hat\sigma_T$ has only polynomial tails. Thus the map
    $s\mapsto e^{-cs}F_s f$ must be integrable against the law of
    $\hat\sigma_T$.  Theorem~\ref{theo:frac-abstract} separates the
    supercritical regime $c\ge\mu_D$, where the deterministic approximation
    error remains of order $h$ at random operational times, from the
    logarithmic regime $\bar\mu_B\le c<\mu_D$, where the same term must be
    truncated in operational time.  We state no theorem below
    $c=\bar\mu_B$.
\end{remark}

The proofs are given in Section~\ref{sec:proof}.

\section{Random Tensor Fields and Kunita Stochastic Flows} \label{sec:tensor}

\subsection{Tensor fields}

Let $X = \R^k$ and $Y = \R^\ell$ be finite-dimensional Euclidean spaces.
For integers $m \ge 0$, denote by
\[
    \mathfrak{L}_m(X,Y)
\]
the space of continuous $m$-linear mappings
\[
    A : X^m \to Y.
\]
In coordinates (for simplicity, when $Y=\R$) an element $A \in \mathfrak{L}_m(X,\R)$ is represented by
\[
    \bigl(a_{i_1,\dots,i_m} \bigr),
    \qquad
    1 \le i_p \le k,
\]
and acts via
\[
    A[x_1,\dots,x_m]
    =
    \sum_{i_1,\dots,i_m=1}^k
        a_{i_1,\dots,i_m} \,
        x_1^{i_1} \cdots x_m^{i_m}.
\]
To avoid ambiguity, we put the arguments of a tensor inside square brackets.

For $A\in\mathfrak{L}_m(X,Y)$, define the operator (spectral) norm
\begin{equation} \label{eq:tensor_norm}
    \|A\| :=
    \sup_{\substack{x_1,\dots,x_m\in X \\ x_i\neq 0}}
    \frac{\|A(x_1,\dots,x_m)\|}{\|x_1\|\cdots\|x_m\|}.
\end{equation}
Because all norms in finite dimensions are equivalent, the operator norm
is comparable to the Frobenius norm
\[
    \|A\|_F^2
    := \sum_{i_1,\dots,i_m} \bigl(a_{i_1,\dots,i_m}\bigr)^2.
\]

For elements $x_1, \dots, x_m \in X$, we denote their tensor product by
$x_1 \otimes \dots \otimes x_m$ or $\bigotimes_{j=1}^m x_j$.
The tensor product of normed linear spaces $X_1, \dots, X_n$ is denoted by
$X_1 \otimes \dots \otimes X_n$ or $\bigotimes_{i=1}^n X_i$.
The $m$-th tensor power of $X$ is denoted by $X^{\otimes m}$.

For an element $u \in X_1 \otimes X_2$, define the \emph{projective tensor norm}
$\|\cdot\|_\pi$ by
\[
    \|u\|_\pi :=
    \inf\Bigl\{
        \sum_{i=1}^N 
        \|x_{1,i}\| \,\|x_{2,i}\|
        \ \Big|\ 
        u = \sum_{i=1}^N x_{1,i} \otimes x_{2,i},\;
        x_{1,i} \in X_1,\; x_{2,i} \in X_2,\; N \in \mathbb{N}
    \Bigr\}.
\]
In the finite-dimensional case, $X_1 \otimes X_2$ equipped with the projective tensor norm $\| \cdot \|_{\pi}$ is complete.

Any tensor $A \in \mathfrak{L}_m(X,Y)$ can be viewed as an element of
$\mathcal{L}(X^{\otimes m}, Y)$.
In fact, the tensor norm on the former coincides with the operator norm on the latter (up to equivalence of norms in finite dimensions).

For a tensor $A \in \mathfrak{L}_m(X,Y)$, define its symmetrisation as 
\[
    \operatorname{Sym}(A)[x_1,\dots,x_m]
    :=
    \frac{1}{m!}
    \sum_{\sigma \in \mathrm{Perm}(m)}
        A[x_{\sigma(1)}, \dots, x_{\sigma(m)}].
\]
We define the symmetric product $x_1 \vee x_2$ as
$\operatorname{Sym}(x_1 \otimes x_2)$, and in general
\[
    x_1 \vee \dots \vee x_m
    :=
    \operatorname{Sym}\bigl(x_1 \otimes \dots \otimes x_m\bigr).
\]
When all arguments coincide, we also write
$v^{\vee m} := v \vee \dots \vee v$.

Fix $m \in \mathbb{N}$.  
For a function $f \in C^m (X)$ (scalar-valued), the tensor
$\nabla f (x) \in \mathfrak{L}_1(X,\R)$ is defined as the first-order directional derivative
\[
    \nabla f (x)[u] :=
    \lim_{\varepsilon \to 0} \frac{f(x + \varepsilon u) - f(x)}{\varepsilon},
\]
and the tensor $\DD^m f (x) \in \mathfrak{L}_m(X,\R)$ is the iterated $m$-th order directional derivative. In coordinates,
\[
    \DD^m f (x)[u_1, \dots, u_m]
    =
    \sum_{i_1, \dots, i_m = 1}^k
        \frac{\partial^m f(x)}{\partial x_{i_1} \dots \partial x_{i_m}}\,
        u^{i_1}_1 \dots u^{i_m}_m.
\]
The tensor $\DD^m f(x)$ is symmetric, so we also write
\[
    \DD^m f(x)\bigl[ x_1 \vee \dots \vee x_m \bigr]
\]
instead of $\DD^m f(x)[x_1,\dots,x_m]$.

\subsection{Kunita’s theory in tensor notation}

We view the derivatives $\DD^jX_t(x,\omega)$ with respect to the initial point as random tensor fields and recall the needed parts of Kunita’s stochastic flow theory \cite{Kunita1984SaintFlour} in tensor notation.

We recall \cite[Theorem~3.1]{Kunita1984SaintFlour} in the notation used here.
\begin{theorem}
    \label{theo:deriv_exist}
    Consider the SDE \eqref{eq:SDE-flow}.  
    Let Assumption~\ref{asm:classical} (items \emph{(1)}–\emph{(2)}) hold.  
    Assume that $\nabla \sigma_j(x)$ is well-defined and (globally) $\alpha$-H\"older continuous for some $\alpha \in (0,1)$ and all $j = 0, \dots, d$.
    Then the process $\nabla X_t(x)$ is well-defined and is locally $\beta$-H\"older continuous in $x$, almost surely, for any $\beta < \alpha$. Furthermore, it is the unique strong solution of the SDE
    \begin{equation}
        \label{eq:jet}
        \nabla X_t(x)
        =
        e
        + \int_0^t \bigl(\nabla b\bigl(X_s(x)\bigr)\bigr)\,\nabla X_s(x)\,ds
        + \sum_{l=1}^d \int_0^t
            \bigl(\nabla \sigma_l\bigl(X_s(x)\bigr)\bigr)\,\nabla X_s(x)\,dW_s^l,
    \end{equation}
    where $e[u] = u$ is the identity mapping on $X$.  
    Moreover, for every $u \in X$ and $p \ge 2$ we have
    $\nabla X_t(x)[u] \in L^p$.
\end{theorem}

For higher-order derivatives we use a tensorial Fa\`a di Bruno formula.
For a finite set $A$, denote by $\mathcal{P}(A,k)$ the set of partitions of $A$
into $k$ disjoint non-empty subsets:
\[
    \mathcal{P}(A,k)
    :=
    \left\{
        P = \{P_1, \dots, P_k\} :
        A = \bigcup_{i=1}^k P_i,\;
        P_i \cap P_j = \varnothing,\ i \ne j,\;
        P_i \neq \varnothing,\ i=1,\dots,k
    \right\}.
\]
For brevity, write $\mathcal{P}(m,k) := \mathcal{P}(\{1,\dots,m\},k)$.

We use the following tensor form of the Fa\`a di Bruno formula; cf.\ \cite[Proposition~3.1]{Licht2024}.
\begin{proposition}
    \label{prop:faadibruno}
    Let $X, Y, Z$ be Banach spaces and let $f \colon X \to Y$,
    $g \colon Y \to Z$.  
    Let $x \in X$ and set $y = f(x)$.  
    Assume that $f$ and $g$ have continuous derivatives up to order $m$ in a neighbourhood of $x$ and $y$, respectively. Then, for $v_1,\dots,v_m\in X$,
    \begin{equation}
    \begin{aligned}
        \DD^m (g \circ f)(x)[ v_1, \dots, v_m]
        &=
        \sum_{k=1}^{m}
        \ \sum_{P \in \mathcal{P}(m,k)}
        \DD^k g\bigl(f(x)\bigr)
        \Bigl[
            \DD^{|P_1|} f(x)\bigl[ \bigvee_{i \in P_1} v_i \bigr],
            \dots,
            \DD^{|P_k|} f(x)\bigl[ \bigvee_{i \in P_k} v_i \bigr]
        \Bigr].
    \end{aligned}
    \end{equation}
\end{proposition}

This is the tensor form of the Fa\`a di Bruno formula.  The classical and multivariate versions are discussed in \cite{Comtet2012,ConstantineSavits1996}; see also \cite{Savits2006} for probabilistic applications.

For $m\in\mathbb N$ and $\alpha\in(0,1)$, let $C^{m,\alpha}$ and $C^{m,\alpha}_g$ denote the classes of maps $f:X\to X$ whose derivatives up to order $m$ are continuous and whose $m$-th derivatives are locally, respectively globally, $\alpha$-H\"older continuous.

We also use \cite[Theorem~3.3]{Kunita1984SaintFlour}.
\begin{theorem}
    Consider the SDE \eqref{eq:SDE-flow}.  
    Let $m \ge 2$ and $\alpha \in (0,1)$, and suppose
    $b, \sigma_j \in C^{m,\alpha}_g$ for all $j=0,\dots,d$.
    Assume that all derivatives of $b$ and $\sigma_j$ up to order $m$ are bounded. 
    Then, for each $t \ge 0$, the mapping $x \mapsto \DD^m X_t(x)$ is of class $C^{m,\beta}$ almost surely, for any $\beta < \alpha$.  
    Furthermore, $\DD^m X_t(x)$ satisfies an SDE of the form
    \begin{equation}
        \label{eq:higher-jet}
        \DD^m X_t(x)
        =
        \int_0^t
            \bigl(\DD^m b\bigl(X_s(x)\bigr)\bigr)
            \bigl[\cdot\bigr]\,ds
        +  \sum_{l=1}^d \int_0^t
            \bigl(\DD^m \sigma_l\bigl(X_s(x)\bigr)\bigr)
            \bigl[\cdot\bigr]\,dW_s^l,
    \end{equation}
    where the dots indicate multilinear dependence on
    $\DD^k X_s(x)$, $k=1,\dots,m$, given explicitly via Proposition~\ref{prop:faadibruno}.
    Moreover, for any $u_1,\dots,u_m \in X$ and $p \ge 2$,
    \[
        \DD^m X_t(x)[u_1,\dots,u_m] \in L^p.
    \]
\end{theorem}

Equation~\eqref{eq:higher-jet} involves the lower-order derivatives
$\DD^k X_t(x)$, $k = 1,\dots, m-1$. In particular, by
Proposition~\ref{prop:faadibruno}, the coefficients can be written in closed form in terms of the tensors $\DD^k b$ and $\DD^k \sigma_j$.

\section{A priori estimates for the sensitivities} \label{sec:sensitivities}

This section estimates the moments of the flow derivatives $\DD^mX_t(x)$ and the corresponding derivatives of the semigroup $F_t$.

\subsection{First-order derivative}

\begin{lemma} \label{lemm:first_order_estim}
    Consider the SDE \eqref{eq:SDE-flow} under Assumption \ref{asm:classical}. 
    Let $p \ge 2$ and $v \in \cX$ with $|v|=1$. Then there exists a constant
    $\mu_{1,p} = \mu_{1,p}(K,d,p) > 0$ such that
    \begin{equation} \label{eq:first_order_estim}
        \E \bigl| \nabla X_t(x)[v] \bigr|^p \;\le\; e^{\mu_{1,p} t},
        \qquad t \ge 0.
    \end{equation}
\end{lemma}

\begin{proof}
    For $p \ge 2$, the map $y \mapsto |y|^p$ on $\cX$ is twice continuously differentiable. 
    Set
    \[
        Y_t := \nabla X_t(x)[v].
    \]
    By Theorem~\ref{theo:deriv_exist}, $Y_t$ satisfies the linear SDE
    \[
        dY_t
        = \bigl(\nabla b(X_t(x))\,Y_t\bigr)\,dt
          + \sum_{l=1}^d \bigl(\nabla \sigma_l(X_t(x))\,Y_t\bigr)\,dW_t^l, \qquad Y_0 = v.
    \]
    Applying It\^o’s formula to $t \mapsto |Y_t|^p$ gives
    \[
        d |Y_t|^p
        = p |Y_t|^{p-2} \langle Y_t, dY_t \rangle
          + \frac{p(p-1)}{2} |Y_t|^{p-2} d\langle Y \rangle_t,
    \]
    where $\langle Y \rangle_t$ is the quadratic variation of $Y$. Taking expectations and using the boundedness of $\nabla b$ and $\nabla\sigma_l$ from Assumption~\ref{asm:classical}, we obtain
    \begin{align*}
        \frac{d}{dt} \E|Y_t|^p
        &=
        p\,\E\Bigl[
            |Y_t|^{p-2} \big\langle Y_t, \nabla b(X_t) Y_t\big\rangle
        \Bigr]
        + \frac{p(p-1)}{2} \sum_{l=1}^d
            \E\Bigl[
                |Y_t|^{p-2}\,\big|\nabla \sigma_l(X_t) Y_t\big|^2
            \Bigr] \\
        &\le
        C_{1,p}\,\E|Y_t|^p,
    \end{align*}
    where $C_{1,p}=C_{1,p}(K,d,p)$ is finite and can be computed explicitly by Young and Cauchy--Schwarz inequalities. Since $Y_0 = v$ and $|v|=1$, Gr\"onwall’s lemma yields
    \[
        \E|Y_t|^p \le e^{C_{1,p} t}, \qquad t\ge0.
    \]
    Setting $\mu_{1,p} := C_{1,p}$ gives \eqref{eq:first_order_estim}.
\end{proof}

\subsection{Higher-order derivatives}

We derive inductive bounds for higher-order derivatives of the flow.

\begin{proposition} \label{prop:higher_moment_derivs}
    Consider the SDE \eqref{eq:SDE-flow} under Assumption \ref{asm:classical}.
    Fix $m \in \mathbb N$ and $p \ge 2$. Then there exist functions
    $A_m(\cdot,p) : [0,\infty) \to [0,\infty)$ such that, for all
    $v_1,\dots,v_m \in \cX$,
    \begin{equation}
        \E \bigl\| \DD^m X_t(x)[v_1,\dots,v_m] \bigr\|^p
        \;\le\; A_m(t,p) \prod_{j=1}^m \|v_j\|^p,
        \qquad t \ge 0.
    \end{equation}
    Moreover, there exist constants $\mu_{m,p}>0$ and $\lambda_{m,p}>0$
    depending only on $K,d,m,p$ such that
    \begin{align}
        A_1(t,p) &\le e^{\mu_{1,p} t}, \\
        A_m(t,p) &\le \lambda_{m,p} \bigl( e^{\mu_{m,p} t } - 1 \bigr),
        \qquad m \ge 2.
    \end{align}
\end{proposition}

\begin{proof}
    We proceed by induction on $m$.  
    The case $m=1$ is given by Lemma~\ref{lemm:first_order_estim}, which yields
    $A_1(t,p) \le e^{\mu_{1,p}t}$.

    Let $m \ge 2$ and suppose the claim holds for all orders
    $1,\dots,m-1$. For brevity, set
    \[
        Z_t^{(m)} := \DD^m X_t(x)[v_1,\dots,v_m].
    \]
    By Kunita’s theory and the tensorial Fa\`a di Bruno formula
    (Proposition~\ref{prop:faadibruno}), $Z_t^{(m)}$ satisfies a linear SDE whose drift and diffusion coefficients at time $s$ are finite linear combinations of terms of the form
    \[
        \DD^k b(X_s(x))
        \Bigl[
            \DD^{|P_1|} X_s(x)[\vee_{i\in P_1} v_i],\dots,
            \DD^{|P_k|} X_s(x)[\vee_{i\in P_k} v_i]
        \Bigr],
    \]
    and similarly with $\sigma_l$ instead of $b$, where
    $1 \le k \le m$ and $P=\{P_1,\dots,P_k\}$ runs over partitions
    $P \in \mathcal P(m,k)$.

    Using the boundedness of the derivatives of $b$ and $\sigma_l$ in Assumption~\ref{asm:classical}, we have for each such term
    \[
        \Bigl\|
            \DD^k \sigma_l(X_s(x))
            \bigl[
                \DD^{|P_1|} X_s(x)[\vee_{i\in P_1} v_i],\dots,
                \DD^{|P_k|} X_s(x)[\vee_{i\in P_k} v_i]
            \bigr]
        \Bigr\|
        \le
        K \prod_{j=1}^k
            \Bigl\|
                \DD^{|P_j|} X_s(x)[\vee_{i\in P_j} v_i]
            \Bigr\|.
    \]
    In this expansion the single-block partition $k=1$ produces the term
    $\DD b(X_s(x))\,Z_s^{(m)}$, and likewise $\DD\sigma_l(X_s(x))\,Z_s^{(m)}$
    in the diffusion part; these are the only terms linear in $Z_s^{(m)}$, they
    are bounded by Assumption~\ref{asm:classical}, and we keep them on the
    homogeneous side.  Every remaining partition has $k\ge2$ blocks, and since
    $\sum_j|P_j|=m$ each such block has size at most $m-1$.
    Applying It\^o’s formula to $t \mapsto \|Z_t^{(m)}\|^p$, taking expectations and dropping martingale terms, we obtain an integral inequality of the form
    \[
        \frac{d}{dt} \E\bigl\|Z_t^{(m)}\bigr\|^p
        \le C_{m,p}\, \E\bigl\|Z_t^{(m)}\bigr\|^p
        + \sum_{k=2}^{m} \sum_{P \in \mathcal P(m,k)}
          \E\Bigl[
              \prod_{j=1}^k
                \bigl\|
                    \DD^{|P_j|} X_t(x)[\vee_{i\in P_j} v_i]
                \bigr\|^p
          \Bigr],
    \]
    where $C_{m,p}$ depends only on $K,d,m,p$.

    Fix a partition $P=\{P_1,\dots,P_k\} \in \mathcal P(m,k)$ with
    $2\le k \le m$.  Applying H\"older's inequality with exponents $(q_1,\dots,q_k)$
    given by $q_j := k$ to the product of $k$ factors,
    \[
        \E\Bigl[
            \prod_{j=1}^k
                \bigl\|\DD^{|P_j|} X_t(x)[\vee_{i\in P_j}v_i]\bigr\|^p
        \Bigr]
        \le
        \prod_{j=1}^k
            \Bigl(
                \E\bigl\|\DD^{|P_j|}X_t(x)[\vee_{i\in P_j}v_i]\bigr\|^{kp}
            \Bigr)^{1/k}.
    \]
    Since $|P_j| \le m-1$ for every block, the induction hypothesis applied
    with exponent $kp \ge 2$ gives
    $\E\|\DD^{|P_j|}X_t(x)[\vee v_i]\|^{kp} \le A_{|P_j|}(t,kp)\prod_{i\in P_j}|v_i|^{kp}$,
    and therefore
    \[
        \frac{d}{dt} A_m(t,p)
        \le C_{m,p}\,A_m(t,p)
        + \tilde C_{m,p}
        \sum_{k=2}^{m}\sum_{P\in\cP(m,k)}
        \prod_{j=1}^k A_{|P_j|}(t,kp)^{1/k},
    \]
    where $\tilde C_{m,p}$ depends only on $K,d,m,p$.  Each
    $A_l(t,kp)$ with $1\le l\le m-1$ is, by the induction hypothesis,
    bounded by $\lambda_{l,kp}(e^{\mu_{l,kp}t}-1)$ for $l\ge2$ and by
    $e^{\mu_{1,kp}t}$ for $l=1$.  Substituting these bounds and applying
    Gr\"onwall's lemma to the resulting linear inequality for $A_m(t,p)$
    yields
    \[
        A_m(t,p)
        \le \lambda_{m,p}\bigl(e^{\mu_{m,p} t}-1\bigr),
    \]
    for suitable constants $\lambda_{m,p},\mu_{m,p} > 0$.
    This completes the induction.
\end{proof}

\subsection{Growth of derivatives of the semigroup}

\begin{lemma} \label{lemm:growh_rate}
    Let $m\in\mathbb N$, $f \in D_m$, and let $F_t$ be the Markov semigroup
    corresponding to the solution $X_t(x)$ of \eqref{eq:SDE-flow}.
    Let Assumption~\ref{asm:classical} hold. Then there exist constants
    $C_m,\bar\mu_m>0$, depending only on $K,d,m$, such that
    \begin{equation} \label{eq:semigroup-deriv-growth}
        \big\| \DD^m F_t f(x) \big\|
        \;\le\; C_m e^{\bar \mu_m t} \, \| f \|_{m},
        \qquad t \ge 0,
    \end{equation}
    uniformly in $x \in \cX$.
    Strong continuity of $(F_t)_{t\ge0}$ in the norm $\|\cdot\|_m$ on the
    space $D_m$ defined in Section~\ref{sec:main} is proved as
    Lemma~\ref{lem:D4-strong-continuity-core} below.
\end{lemma}

\begin{proof}
    Without loss of generality we may assume $|v_j|=1$ for all $j=1,\dots,m$.
    By definition and dominated convergence,
    \[
        \DD^m F_t f(x)[v_1,\dots,v_m]
        = \DD^m \E\bigl[f(X_t(x))\bigr][v_1,\dots,v_m].
    \]
    Applying the tensorial Fa\`a di Bruno formula
    (Proposition~\ref{prop:faadibruno}) with $g=f$ and $f(\cdot)=X_t(\cdot)$ yields
    \begin{multline*}
        \bigl|\DD^m F_t f(x)[v_1,\dots,v_m]\bigr|
        \le \sum_{k=1}^m \sum_{P \in \mathcal P(m,k)}
        \E\Bigl[
            \bigl\|\DD^k f(X_t(x))\bigr\|
            \prod_{j=1}^k
                \bigl\|
                    \DD^{|P_j|} X_t(x)[\vee_{l\in P_j} v_l]
                \bigr\|
        \Bigr] \\
        \le \|f\|_m
        \sum_{k=1}^m \sum_{P \in \mathcal P(m,k)}
            \prod_{j=1}^k
                \Bigl(
                    \E\bigl\|
                        \DD^{|P_j|} X_t(x)[\vee_{l\in P_j} v_l]
                    \bigr\|^{k}
                \Bigr)^{1/k}.
    \end{multline*}
    Using Proposition~\ref{prop:higher_moment_derivs} with $p=k$ and $|v_i|=1$, we obtain
    \[
        \E\bigl\|
            \DD^{|P_j|} X_t(x)[\vee_{l\in P_j} v_l]
        \bigr\|^{k}
        \le A_{|P_j|}(t,k),
    \]
    so that
    \[
        \bigl|\DD^m F_t f(x)[v_1,\dots,v_m]\bigr|
        \le \|f\|_m
        \sum_{k=1}^m \sum_{P \in \mathcal P(m,k)}
            \prod_{j=1}^k A_{|P_j|}(t,k)^{1/k}.
    \]
    By Proposition~\ref{prop:higher_moment_derivs}, $A_1(t,k)\le e^{\mu_{1,k} t}$ and
    $A_{l}(t,k)\le \lambda_{l,k}(e^{\mu_{l,k} t}-1)$ for $l\ge2$.  
    Each partition $P\in\mathcal P(m,k)$ with $k<m$ contains at least one block of size $\ge2$, so we obtain an upper bound of the form
    \[
        \bigl|\DD^m F_t f(x)[v_1,\dots,v_m]\bigr|
        \le C_m e^{\bar\mu_m t} \|f\|_m,
    \]
    for suitable constants $C_m,\bar\mu_m>0$ depending only on $K,d,m$.
    Taking the supremum over all unit vectors $v_1,\dots,v_m$ and $x\in\cX$
    yields \eqref{eq:semigroup-deriv-growth}.
\end{proof}

\subsection{Renorming and weighted estimates}

The previous lemma gives growth bounds, but not a quasi-contraction in the natural norm.  We use the equivalent renorming below.

\begin{proposition}
    \label{prop:renorming}
    Let $(P_t)_{t\ge 0}$ be a strongly continuous semigroup on a Banach space $(X,\|\cdot\|)$.
    Assume there exist constants $M\ge 1$ and $\omega\in\mathbb{R}$ such that
    \begin{equation}\label{eq:semigroup-bound}
        \|P_t\|_{X\to X} \;\le\; M e^{\omega t}, 
        \qquad t\ge 0.
    \end{equation}
    Then, for every $\varepsilon>0$, there exists an equivalent norm 
    $\|\cdot\|_{*}$ on $X$ such that
    \begin{equation}\label{eq:quasi-contract}
        \|P_t x\|_{*} \;\le\; e^{(\omega+\varepsilon)t}\, \|x\|_{*},
        \qquad x\in X,\; t\ge 0.
    \end{equation}
    In particular, with respect to $\|\cdot\|_{*}$, the semigroup $(P_t)_{t\ge0}$ is quasi-contractive.
\end{proposition}
    
\begin{proof}
    Fix $\varepsilon>0$ and define
    \begin{equation}
        \|x\|_{*}
        := \sup_{t\ge 0} e^{-(\omega+\varepsilon)t}\, \|P_t x\|.
    \end{equation}
    $\|\cdot\|_{*}$ is a norm. Moreover,
    \[
        \|x\|_{*}
        \;\ge\; \|P_0 x\| = \|x\|,
    \]
    and by \eqref{eq:semigroup-bound},
    \[
        \|x\|_{*}
        \le \sup_{t\ge0} e^{-(\omega+\varepsilon)t} M e^{\omega t} \|x\|
        = M \|x\|.
    \]
    Thus $\|\cdot\|_{*}$ is equivalent to $\|\cdot\|$.

    For $s\ge0$,
    \begin{align*}
        \|P_s x\|_{*}
        &= \sup_{t\ge 0} e^{-(\omega+\varepsilon)t}\, \|P_t(P_s x)\|
         = \sup_{t\ge 0} e^{-(\omega+\varepsilon)t}\, \|P_{t+s} x\| \\
        &= \sup_{r\ge s} e^{-(\omega+\varepsilon)(r-s)}\, e^{-(\omega+\varepsilon)s} \|P_r x\|
         \le e^{(\omega+\varepsilon)s}\, \sup_{r\ge 0} e^{-(\omega+\varepsilon)r}\|P_r x\| \\
        &= e^{(\omega+\varepsilon)s}\, \|x\|_{*},
    \end{align*}
    which gives \eqref{eq:quasi-contract}.
\end{proof}

We need the following polynomial Lyapunov bound.
For an integer $q \ge 1$ define
\[
    V_q(x) := 1 + |x|^{2q}, \qquad x \in \R^k,
\]
and the weighted norm
\[
    \|f\|_{V_q} := \sup_{x \in \R^k} \frac{|f(x)|}{V_q(x)}.
\]

\begin{lemma} \label{lem:weighted_estim_tensor_poly}
    Let $X_t^x$ solve \eqref{eq:SDE-flow} under Assumption~\ref{asm:classical}, and let $V_q(x)=1+|x|^{2q}$ with $q\in\mathbb N$.  Then there are constants $C_q,D_q>0$, depending only on $q,K,d$, such that
    \begin{equation} \label{eq:LVk-bound}
        L V_q(x) \;\le\; C_q \, V_q(x),
        \qquad x \in \cX,
    \end{equation}
    and
    \begin{equation} \label{eq:L2Vk-bound}
        \bigl|L^2 V_q(x)\bigr|
        \;\le\; D_q \, V_q(x),
        \qquad x \in \cX.
    \end{equation}
    Consequently, for every measurable $f$ with $\|f\|_{V_q}<\infty$,
    \[
        \|F_t f\|_{V_q} \;\le\; e^{C_q t} \, \|f\|_{V_q}, \qquad t \ge 0.
    \]
    One can take
    \[
        C_q := (4qK + 6q^2 K^2)\, d,
        \qquad
        D_q := C_* \, q^4 K^4 \, d^2
    \]
    for a universal constant $C_* > 0$ independent of $q,K,d$, and $k$.
\end{lemma}

\begin{proof}
    We first estimate $L V_q$. For
    \[
        V_q(x)=1+|x|^{2q},
    \]
    one has
    \[
        \nabla V_q(x)=2q\,|x|^{2q-2}x,
    \]
    and
    \[
        \DD^2 V_q(x)[u,v]
        =
        2q\,|x|^{2q-2}\langle u,v\rangle
        +4q(q-1)\,|x|^{2q-4}\langle x,u\rangle\langle x,v\rangle.
    \]
    Hence
    \begin{equation}
        |\nabla V_q(x)| \le C_q(1+|x|^{2q-1}),
        \qquad
        \|\DD^2 V_q(x)\| \le C_q(1+|x|^{2q-2}).
    \end{equation}
    Using the linear growth of $b,\sigma$ gives
    \[
        L V_q(x)
        =
        \nabla V_q(x)[b(x)]
        + \frac12 \sum_{j=1}^d \DD^2 V_q(x)[\sigma_j(x),\sigma_j(x)]
        \le C_q V_q(x),
    \]
    which is \eqref{eq:LVk-bound}. By Dynkin's formula and Gr\"onwall's lemma,
    \[
        \E[V_q(X_t^x)] \le e^{C_q t}V_q(x),
    \]
    and therefore
    \[
        \|F_t f\|_{V_q} \le e^{C_q t}\|f\|_{V_q}.
    \]

    To estimate $L^2V_q$, write
    \[
        g(x):=LV_q(x).
    \]
    For every multiindex $\alpha$ with $|\alpha|\le 4$,
    \begin{equation}
        |\partial^\alpha V_q(x)|
        \le
        C(q)\bigl(1+|x|^{2q-|\alpha|}\bigr).
    \end{equation}
    Because $b,\sigma_j$ have bounded derivatives up to order $m\ge4$ and at most linear growth, every derivative of $g$ of order at most two is a finite sum of products involving derivatives of $V_q$ of order at most four, bounded derivatives of $b,\sigma_j$, and at most one undifferentiated factor of $b$ or two undifferentiated factors of $\sigma_j$. Consequently, differentiating lowers the polynomial degree exactly as expected, and there exist constants $C_0,C_1,C_2>0$ such that
    \begin{equation}\label{eq:g-derivative-bounds}
        |g(x)| \le C_0(1+|x|^{2q}),
        \qquad
        |\nabla g(x)| \le C_1(1+|x|^{2q-1}),
        \qquad
        \|\DD^2 g(x)\| \le C_2(1+|x|^{2q-2}).
    \end{equation}
    Applying $L$ once more gives
    \[
        L^2V_q(x)
        =
        \nabla g(x)[b(x)]
        + \frac12 \sum_{j=1}^d \DD^2 g(x)[\sigma_j(x),\sigma_j(x)].
    \]
    Using \eqref{eq:g-derivative-bounds} and the linear growth of $b,\sigma$,
    \[
    \begin{aligned}
        |L^2V_q(x)|
        &\le
        |b(x)|\,|\nabla g(x)|
        + \frac12 \|\DD^2 g(x)\| \sum_{j=1}^d |\sigma_j(x)|^2 \\
        &\le
        K(1+|x|)\,C_1(1+|x|^{2q-1})
        + \frac12 dK^2(1+|x|)^2\,C_2(1+|x|^{2q-2}) \\
        &\le
        D_q(1+|x|^{2q})
        =
        D_q V_q(x),
    \end{aligned}
    \]
    which proves \eqref{eq:L2Vk-bound}. Enlarging the universal constant $C_*$ if needed, one may take
    \[
        D_q \le C_* q^4 K^4 d^2.
    \]

\end{proof}

Combining Lemma~\ref{lemm:growh_rate}, Lemma~\ref{lem:weighted_estim_tensor_poly}, Lemma~\ref{lem:D4-strong-continuity-core}, and Proposition~\ref{prop:renorming}, we obtain the following.

\begin{corollary} \label{coro:norm_estim}
    For each $m\in\mathbb N$ there exists an equivalent norm $\|\cdot\|_m^{*}$ on $D_m$ such that
    \[
        \|F_t f\|_m^{*}
        \;\le\; e^{\mu_m t} \|f\|_m^{*}, \qquad t\ge0,
    \]
    where $\mu_m \ge 0$ depends only on $K,d,m$.
\end{corollary}

\section{Function Spaces and Semigroups for Unbounded Functions}
\label{sec:weighted-spaces}

We record the weighted semigroup facts needed in Section~\ref{sec:proof}.  The point is to work on Banach spaces that are invariant under $F_t$ and on which $F_tf\to f$ strongly as $t\downarrow0$.  Under these conditions the generator is well defined on a dense domain; see \cite[Theorem~2.4]{Kolokoltsov2010Nonlinear}.

For a weight $V\ge1$ with $V(x)\to\infty$ as $|x|\to\infty$, put
\[
    C_{V,\infty}(\cX)
    :=
    \Bigl\{
        f\in C(\cX):
        \|f\|_V:=\sup_{x\in\cX}\frac{|f(x)|}{V(x)}<\infty,\quad
        f/V\in C_\infty(\cX)
    \Bigr\}.
\]

We use the following weighted Feller consequence of \cite[Theorem~5.2]{Kolokoltsov2010Nonlinear}.
\begin{lemma} \label{lem:strong_continuity_weighted_generic}
    Let $(X_t^x)_{t\ge0}$ satisfy Assumption~\ref{asm:classical}, and suppose that a twice continuously differentiable weight $V\ge1$ satisfies
    \[
        V(x)\xrightarrow[|x|\to\infty]{}\infty,
        \qquad
        LV\le cV
    \]
    for some $c\ge0$.  Then $(F_t)_{t\ge0}$ is a bounded strongly continuous semigroup on $C_{V,\infty}(\cX)$ and
    \[
        \|F_t f\|_V\le e^{ct}\|f\|_V,\qquad
        \|F_t f-f\|_V\xrightarrow[t\downarrow0]{}0.
    \]
\end{lemma}

\begin{proof}
    The Lyapunov bound and Dynkin's formula give $\E V(X_t^x)\le e^{ct}V(x)$.  The cited weighted Feller theorem then gives invariance of $C_{V,\infty}(\cX)$ and strong continuity in the weighted norm.
\end{proof}

The next lemma records the escape-to-infinity property of the diffusion that
underlies the vanishing-at-infinity bookkeeping in
Section~\ref{sec:proof}.  It is the only place where we use that the process,
started far away, does not return to a fixed compact set within a bounded time.

\begin{lemma} \label{lem:escape}
    Let Assumption~\ref{asm:classical} hold.  Then, for every compact
    $K\subset\cX$ and every $T>0$,
    \[
        \sup_{0\le t\le T}P_x(X_t\in K)\xrightarrow[\;|x|\to\infty\;]{}0 .
    \]
\end{lemma}
\begin{proof}
    Fix $p>0$ and consider the bounded weight
    $w_{-p}(x)=(1+|x|^2)^{-p/2}$, which vanishes at infinity.  Writing
    $\nabla w_{-p}(x)=-p(1+|x|^2)^{-p/2-1}x$ and
    $\DD^2 w_{-p}(x)=-p(1+|x|^2)^{-p/2-1}I+p(p+2)(1+|x|^2)^{-p/2-2}\,x\otimes x$,
    the linear growth $|b(x)|,|\sigma_\ell(x)|\le K(1+|x|)$ together with
    $|x|(1+|x|)\le C(1+|x|^2)$ gives, after collecting terms, a constant
    $c_p=c_p(K,d,p)$ with
    \[
        L w_{-p}(x)\le c_p\,w_{-p}(x),\qquad x\in\cX.
    \]
    Since $w_{-p}$ is bounded with bounded derivatives, Dynkin's formula and
    Gr\"onwall's lemma yield $\E w_{-p}(X_t(x))\le e^{c_pt}w_{-p}(x)$.  On the
    compact set $K$ the weight is bounded below, so
    $\mathbf 1_K\le(\inf_K w_{-p})^{-1}w_{-p}$, and therefore
    \[
        \sup_{0\le t\le T}P_x(X_t\in K)
        \le\frac{e^{c_pT}}{\inf_K w_{-p}}\,w_{-p}(x)
        \xrightarrow[\;|x|\to\infty\;]{}0 .
        \qedhere
    \]
\end{proof}

\section{Proofs of the main results} \label{sec:proof}

    \subsection{Proof of Theorem \ref{prop:uniform}}

    We use the following variant of \cite[Proposition~1]{Kolokoltsov2023Rates}.
    \begin{proposition} \label{prop:1}
        Let $F_t = e^{tL}$ be a strongly continuous bounded semigroup on a Banach space $(B,\|\cdot\|_B)$:
        \[
            \max_{s \in [0,t]} \|F_s\|_{B \to B} \leq e^{\bar m t}.
        \]
        
        Let $L$ be the generator of $F_t$, and let $D \subset \operatorname{dom} L$ be another Banach space with $\|\cdot\|_D \geq \|\cdot\|_B$ such that
        \[
        \|Lf\|_B \leq l \|f\|_D \, \forall f \in D.
        \]
        Assume also that $F_t$ is bounded in $D$:
        \[
        \max_{s \in [0,t]} \|F_s\|_{D \to D} \leq e^{m t},
        \]
        for some $m \geq 0$.

        Let $U_h$ be a family of linear bounded operators in $B$ with 
        \[
          \| U_h \|_{B \to B} \leq e^{q h},
          \qquad q\ge0,
        \]
        and let
        \begin{align} \label{eq:prop1aneq}
        \left\| \left( \frac{U_h - 1}{h} - L \right) f \right\|_B &\leq \epsilon_h \|f\|_D, \\
        \label{eq:prop1aneq2}
        \left\| \left( \frac{F_h - 1}{h} - L \right) f \right\|_B &\leq \chi_h \|f\|_D, 
        \end{align}
        where $\epsilon_h,\chi_h>0$ are monotone and tend to zero as $h\to0$.
        
        Then
        \begin{equation} \label{eq:prop1.1}
        \begin{aligned}
            \sup_{s \leq t} \left\| (U_h)^{\lfloor s/h\rfloor} f - F_s f \right\|_B
            &\leq
            (\chi_h + \epsilon_h) \|f\|_D
            \int_0^t e^{ms + q (t-s)} \, ds  \\
            &\quad
            + e^{\bar m t} h (l + \epsilon_h) \|f\|_D .
        \end{aligned}
        \end{equation}

        For the killed semigroups with rate $c>0$,
        \begin{equation} \label{eq:prop1.2}
        \begin{aligned}
            \sup_{s \leq t}
            \left\| e^{-c\lfloor s/h\rfloor h } (U_h)^{\lfloor s/h\rfloor} f
            - e^{-cs} F_s f \right\|_B
            &\leq
            (\chi_h + \epsilon_h) \|f\|_D
            \int_0^t e^{(m-c)s + (q-c) (t-s)} \, ds  \\
            &\quad
            + e^{(\bar m - c) t} h (l + \epsilon_h + c) \|f\|_D .
        \end{aligned}
        \end{equation}
        \end{proposition}
        \begin{remark}
            Contrary to \cite[Proposition 1]{Kolokoltsov2023Rates}, we do not require $D$ to be dense in $B$.
        \end{remark}
        
        \begin{proof}
         The estimate \eqref{eq:prop1aneq2} is equivalent to
        \[
        \left\| \frac{1}{h} \int_0^h (F_s - 1)L f\, ds \right\|_B \leq \chi_h \|f\|_D,
        \]
        because
        \[
        F_t f - f = t L f + \int_0^t (F_s - 1) L f\, ds.
        \]
        
        By \eqref{eq:prop1aneq} and
        \[
        U_h - F_h = (U_h - 1) -  \int_0^h F_s L\, ds = (U_h - 1) -  h L -  \int_0^h (F_s - 1) L\, ds,
        \]
        we have
        \begin{equation*} 
            \| U_h f -  F_h f \|_B \leq h (\epsilon_h + \chi_h) \|f\|_D. 
        \end{equation*}
    
        Since
        \[
        U^{k}_h  f -  F_{kh} f = U^{k}_h  f - F_h^k f,
        \]
        the telescoping expansion gives
        \begin{equation*} 
        U^{k}_h  f - F_h^k f
        = U^{k}_h  f - U^{(k-1)}_h  F_h f + U^{(k-1)}_h  F_h f - U^{(k-2)}_h  F_h^2 f + \dots + U_h F_h^{k-1} f -  F_h^k f,
        \end{equation*}
        and, since for $a \geq 1$ and $b \geq 0$
        \[
        \begin{aligned}
         \| U^a_h F^b_h f - U^{a-1}_h F^{b+1}_h f \|_B
         &\leq
         \|U^{a-1}_h \|_{B \to B}
         \| (U_h - F_h) F^{b}_h f \|_B   \\
         &\leq
         \|U_h \|_{B \to B}^{a-1}
         h (\epsilon_h + \chi_h)
         \| F_h \|_{D \to D}^{b} \|f\|_D .
        \end{aligned}
        \]
        hence
        \[
        \| U^{k}_h  f -  F_h^k f \|_B \leq h(\epsilon_h + \chi_h) \|f\|_D (e^{(k-1)qh} + e^{(k-2) qh} e^{mh} + \dots + e^{(k-1)mh}),
        \]
        implying that
        \begin{equation} \label{eq:prop1d}
            \max_{k \leq [t/h]} \| U^{k}_h  f -  F_h^k f \|_B
            \leq
            (\epsilon_h + \chi_h) \|f\|_D
            \int_0^t e^{ms + q(t-s)}ds.
        \end{equation}
            
        For arbitrary times,
        \[
           \| U^{[t/h]}_h f - F_t f \|_B \leq \| U^{[t/h]}_h f - F_h^{[t/h]} f\|_B + \| F_{h[t/h]} f - F_t f \|_B.
        \]
        For $\delta \leq h$, we have
        \begin{equation} \label{eq:prop1c}
        \begin{aligned}
            \| F_\delta f - f \|_B
            &=
            \left\| Lf \delta
            + \int_0^\delta (F_s- 1) L f ds \right\|_B  \\
            &\leq \delta ( l + \chi_h) \|f\|_D .
        \end{aligned}
        \end{equation}
        With $\delta = t - h[t/h]$,
        \begin{multline*} 
            \| F_{h[t/h]} f - F_t f \|_B
            = \| F_{h[t/h]} (f - F_{\delta} f)\|_B
            \leq \|F_{h[t/h]}\|_{B\to B}\, \|F_{\delta} f - f\|_B
            \leq  e^{\bar m t} \delta ( l + \chi_h) \|f\|_D  \\
            \leq e^{\bar m t} h (l + \chi_h) \|f\|_D.
        \end{multline*}
        This proves \eqref{eq:prop1.1}.  For \eqref{eq:prop1.2}, set
        \[
            \hat F_t := e^{-c t} F_t, \qquad \hat U_h := e^{-c h} U_h.
        \]
        The semigroup $\hat F_t$ has generator $\hat L:=L-c$, and
        \[
            \| \hat F_t \|_{B \to B} = e^{-ct} \|  F_t  \|_{B \to B} \leq e^{(\bar m - c)t},
        \]
        and 
        \[
            \| \hat F_t \|_{D \to D} = e^{-ct} \|  F_t  \|_{D \to D} \leq e^{( m - c)t}.
        \]
        Also,
        \[
            \| \hat L f \|_B \leq \| Lf\|_B + \|cf\|_B \leq (l + c) \|f\|_D.
        \]

        For $\delta \leq h$, we have
        \[
            \hat F_\delta f - f = e^{-c \delta} F_\delta f - f = e^{-c\delta} (F_\delta f - f) - (1 - e^{-c\delta}) f.  
        \]
        By \eqref{eq:prop1c} and Bernoulli's inequality,
        \[
            \|\hat F_\delta f - f \| \leq \delta ( l + c + \chi_h) \|f\|_D.
        \]
        Taking supremum for $\delta \leq h$, we get
        \[
           \sup_{0 \leq \delta \leq h} \|\hat F_\delta f - f \| \leq h ( l + c + \chi_h) \|f\|_D.
        \]
        Combining this bound with \eqref{eq:prop1d} gives \eqref{eq:prop1.2}.
    \end{proof}

    We apply Proposition~\ref{prop:1} with $V(x):=V_1(x)$, noting that $V\le2W$.  The preceding Lyapunov estimate gives an admissible base growth rate $\bar m=4K+K^2$.  We next estimate $L$ and $L^2$ in the weighted norms used below.

    \begin{lemma} \label{lem:L-weighted}
        Let Assumption~\ref{asm:classical} hold. Let $L$ be the generator of the diffusion
        corresponding to \eqref{eq:SDE-flow}, acting on $f\in D_2$ by
        \[
            Lf(x)  = \nabla f(x)[b(x)] + \frac12 \sum_{l=1}^d \DD^2 f(x)\bigl[\sigma_l(x),\sigma_l(x)\bigr].
        \]
        Then for a constant $l = 2 K + d K^2 > 0$ and for all  $f \in D_2$,
        \begin{equation} \label{eq:L-weighted-estimate}
            \|Lf\|_V \;\le\; l \,\|f\|_2.
        \end{equation}
        Furthermore, if $f \in D_4$, then there exists a constant $l' = l'(K,d) >0$ such that
        \begin{equation} \label{eq:L2-weighted-estimate}
            \|L^2 f\|_W \;\le\; l' \,\|f\|_4.
        \end{equation}
    \end{lemma}
    \begin{proof}
        The first bound follows directly from the definition of $L$:
        \[
            |Lf(x)|
            \le
            \|\nabla f(x)\|\,|b(x)|
            + \frac12 \|\DD^2 f(x)\| \sum_{l=1}^d |\sigma_l(x)|^2.
        \]
        Using $|b(x)|\le K(1+|x|)$, $|\sigma_l(x)|\le K(1+|x|)$, and
        \[
            1+|x| \le 1+|x|^2 = V(x),\qquad
            (1+|x|)^2 \le 2V(x),
        \]
        we obtain
        \[
            |Lf(x)|
            \le
            \Bigl(K + dK^2\Bigr)V(x)\|f\|_2
            \le
            \bigl(2K+dK^2\bigr)V(x)\|f\|_2,
        \]
        which proves \eqref{eq:L-weighted-estimate}.

        For the second estimate, introduce the diffusion tensor
        \[
            a(x):=\sum_{l=1}^d \sigma_l(x)\vee \sigma_l(x).
        \]
        Then
        \[
            Lf = \nabla f[b] + \frac12 \DD^2 f[a].
        \]
        By differentiating once and twice, we obtain for $u,v\in\cX$
        \[
        \begin{aligned}
            \nabla(Lf)[u]
            &= \DD^2 f[u,b]
            + \nabla f[\nabla b[u]]
            + \frac12 \DD^3 f[u,a]
            + \frac12 \DD^2 f[\nabla a[u]],
        \end{aligned}
        \]
        and
        \[
        \begin{aligned}
            \DD^2(Lf)[u,v]
            &=
            \DD^3 f[u,v,b]
            + \DD^2 f[u,\nabla b[v]]
            + \DD^2 f[v,\nabla b[u]]
            + \nabla f[\DD^2 b[u,v]] \\
            &\quad
            + \frac12 \DD^4 f[u,v,a]
            + \frac12 \DD^3 f[u,\nabla a[v]]
            + \frac12 \DD^3 f[v,\nabla a[u]]
            + \frac12 \DD^2 f[\DD^2 a[u,v]].
        \end{aligned}
        \]
        Hence
        \[
            L^2 f
            =
            \nabla(Lf)[b]
            + \frac12 \DD^2(Lf)[a]
        \]
        is a finite sum of terms involving $\DD^r f$ for $r=1,2,3,4$ applied to tensors built from
        \[
            b,\ a,\ \nabla b,\ \DD^2 b,\ \nabla a,\ \DD^2 a.
        \]
        Under Assumption~\ref{asm:classical},
        \[
            |b(x)|\le K(1+|x|),\qquad
            \|a(x)\|\le dK^2(1+|x|)^2,
        \]
        while $\|\nabla b(x)\|+\|\DD^2 b(x)\|\le K$ and
        \[
            \|\nabla a(x)\|
            \le 2dK^2(1+|x|),\qquad
            \|\DD^2 a(x)\|
            \le 4dK^2(1+|x|).
        \]
        Therefore every term in $L^2f(x)$ is bounded by
        \[
            C(K,d)\,(1+|x|)^m\,\|f\|_4,
            \qquad m\le 4.
        \]
        Summing all contributions and using $(1+|x|)^m\le C_m W(x)$ for $m\le4$, we get
        \[
            |L^2f(x)| \le l'(K,d)\,W(x)\,\|f\|_4,
        \]
        which is \eqref{eq:L2-weighted-estimate}.

    \end{proof}

    \begin{lemma} \label{lem:D4-density}
        For every integer $l \ge 0$, the space $C_c^\infty(\cX)$ is dense
        in $D_l$ with respect to $\|\cdot\|_l$.
    \end{lemma}
    \begin{proof}
        Let $f \in D_l$. Choose $\chi \in C_c^\infty(\cX)$ with $0\le\chi\le1$,
        $\chi=1$ on the unit ball and $\chi=0$ outside the ball of radius~$2$,
        and set $\chi_R(x) := \chi(x/R)$, $f_R := \chi_R f$. By Leibniz' rule,
        for $0 \le j \le l$,
        \[
            \DD^j(f - f_R)
            = (1-\chi_R)\DD^j f
            + \sum_{\ell=1}^{j} C_{j,\ell}\,\DD^\ell(1-\chi_R)\,\DD^{j-\ell}f .
        \]
        The first term tends to zero in sup norm because
        $\|\DD^j f\| \in C_\infty(\cX)$. For $1 \le \ell \le j$,
        $\DD^\ell(1-\chi_R)$ is supported on $\{R \le |x| \le 2R\}$ with
        $\|\DD^\ell\chi_R\|_\infty \le C_\ell R^{-\ell}$. When $j - \ell \ge 1$,
        the cross term is bounded by
        $C_\ell R^{-\ell}\sup_{|x|\ge R}\|\DD^{j-\ell}f\| \to 0$
        because $\|\DD^{j-\ell}f\| \in C_\infty(\cX)$. When $j = \ell \ge 1$,
        write $x = \rho e$ with $|e|=1$; by the fundamental theorem of
        calculus,
        \[
            \frac{|f(\rho e)|}{\rho}
            \le \frac{|f(0)|}{\rho}
              + \frac{1}{\rho}\int_0^\rho \|\nabla f(s e)\|\,ds.
        \]
        Since $\|\nabla f\| \in C_\infty(\cX)$, the right-hand side tends to
        zero uniformly in the direction $e$, and therefore
        $R^{-1}\sup_{|x|\ge R}|f(x)| \to 0$. For $j = \ell \ge 2$ a similar
        estimate using $f/W \in C_\infty$ and $W(x) = 1+|x|^4$ gives
        $R^{-j}\sup_{|x|\ge R}|f(x)| \to 0$. Hence $f_R \to f$ in $\|\cdot\|_l$.

        For fixed $R$, mollify $f_R$ by a standard mollifier $\rho_\varepsilon$.
        The functions $f_{R,\varepsilon} := \rho_\varepsilon * f_R$ belong to
        $C_c^\infty(\cX)$ and converge to $f_R$ in the usual $C^l$ norm on
        $\overline{B_{2R}}$, hence in $\|\cdot\|_l$. This proves density.
    \end{proof}

    \begin{lemma} \label{lem:D4-compact-core}
        Let Assumption~\ref{asm:classical} hold with $m \ge l$, and let
        $g \in C_c^\infty(\cX)$. Then for every compact $K \subset \cX$ and
        every $0 \le r \le l$,
        \[
            \sup_{x \in K}\bigl\|\DD^r F_t g(x) - \DD^r g(x)\bigr\| \to 0,
            \qquad t \downarrow 0.
        \]
    \end{lemma}
    \begin{proof}
        The case $r = 0$ follows from continuity of $g$ and the uniform-on-compacts
        convergence $X_t(x) \to x$ in probability for SDEs with linear-growth
        Lipschitz coefficients, combined with the moment bound
        $\sup_{x \in K}\E|X_t(x)|^p < \infty$ ($p\ge1$) and Vitali's theorem.
        Fix $1 \le r \le l$, write $v = (v_1,\dots,v_r)$ with $|v_i| \le 1$,
        and decompose the Fa\`a di Bruno representation
        \[
            \DD^r F_t g(x)[v]
            = S_t(x)[v] + R_t(x)[v],
        \]
        with
        \[
            S_t(x)[v]
            := \E_x\bigl[\DD^r g(X_t(x))[\nabla X_t(x)v_1,\dots,\nabla X_t(x)v_r]\bigr]
        \]
        the singleton-partition part and $R_t(x)[v]$ the sum over partitions
        $P \in \cP(r,q)$, $q < r$, each containing at least one block of size
        $|P_i| \ge 2$.

        For $R_t$, every summand contains a higher jet $\DD^{|P_i|}X_t(x)$
        with $|P_i| \ge 2$. H\"older's inequality with $q+1$ factors
        combined with Proposition~\ref{prop:higher_moment_derivs} (the
        $A_{|P_i|}(t,\cdot)$ bounds vanish as $t\downarrow 0$ with rate
        $(e^{ct}-1)^{1/p}$) gives
        \[
            \sup_{x \in K} \|R_t(x)\|
            \le C_{g,K}\bigl(e^{ct}-1\bigr)^{1/p}\,e^{ct} \to 0,
            \qquad t \downarrow 0,
        \]
        for some $p > 1$ depending on $r$.

        For the singleton part, expand
        $\bigotimes_i(I + (\nabla X_t(x) - I))v_i$ and split
        \[
            S_t(x)[v] - \DD^r g(x)[v]
            = \bigl(\E_x[\DD^r g(X_t(x))[v]] - \DD^r g(x)[v]\bigr) + B_t(x)[v],
        \]
        where $B_t(x)[v]$ collects the $2^r - 1$ non-identity multilinear
        terms. The first difference tends to zero uniformly for $x \in K$ by
        continuity of $\DD^r g$ and Vitali's theorem as above. For $B_t$,
        multilinearity and boundedness of $\DD^r g$ give
        \[
            |B_t(x)[v]|
            \le C_g\,\E_x\bigl[(1 + \|\nabla X_t(x)\|^{r-1})\,\|\nabla X_t(x) - I\|\bigr].
        \]
        The SDE \eqref{eq:jet} for $\nabla X_t(x) - I$ starts from zero. The
        Burkholder--Davis--Gundy inequality, together with the linear-growth
        moment bound from Lemma~\ref{lemm:first_order_estim}, gives
        $\sup_{x \in K}\E_x\|\nabla X_t(x) - I\|^p \le C_K t^{p/2}$ for
        $0 < t \le 1$. Hence
        $\sup_{x \in K}|B_t(x)[v]| \le C_{g,K} t^{1/2} \to 0$.
    \end{proof}

    \begin{lemma} \label{lem:D4-strong-continuity-core}
        Let Assumption~\ref{asm:classical} hold with $m \ge l$. Then
        $F_t D_l \subset D_l$ for every $t \ge 0$, and $(F_t)_{t\ge0}$ is
        strongly continuous on $D_l$ in the norm $\|\cdot\|_l$.
    \end{lemma}
    \begin{proof}
        \emph{Step 1 ($F_tD_l \subset D_l$).}
        Let $f \in D_l$. Lemma~\ref{lemm:growh_rate} gives
        $\sup_x\|\DD^jF_t f(x)\| \le C_{j,t}\|f\|_j$ for $1 \le j \le l$,
        so $F_t f \in C^l(\cX)$ with bounded derivatives.
        Since $D_l \subset C_{W,\infty}(\cX)$, the weighted Feller property
        (Lemma~\ref{lem:strong_continuity_weighted_generic}) gives
        $F_t f / W \in C_\infty(\cX)$.

        To prove $\|\DD^j F_t f\| \in C_\infty(\cX)$, choose
        $g_n \in C_c^\infty(\cX)$ with $g_n \to f$ in $\|\cdot\|_l$
        (Lemma~\ref{lem:D4-density}). By Lemma~\ref{lemm:growh_rate},
        $\|\DD^j F_t (g_n - f)\|_\infty \le C_{l,t}\|g_n - f\|_l \to 0$,
        so $\DD^j F_t f$ is a uniform limit of bounded continuous functions.
        For each $g_n$, the Fa\`a di Bruno representation
        \[
            \DD^j F_t g_n(x)
            = \E\sum_{q=1}^j\sum_{P\in\cP(j,q)}
              \DD^q g_n(X_t(x))
              \bigl[\DD^{|P_1|}X_t(x),\dots,\DD^{|P_q|}X_t(x)\bigr]
        \]
        is supported, via $\DD^q g_n(X_t(x))$, on $\{X_t(x) \in \mathrm{supp}\,g_n\}$.
        By Lemma~\ref{lem:escape}, for any compact $K \subset \cX$,
        \[
            \sup_{0\le t\le1}P_x(X_t \in K) \to 0 \qquad \text{as } |x|\to\infty .
        \]
        Applying H\"older's
        inequality and combining with the moment estimates of
        Proposition~\ref{prop:higher_moment_derivs} yields
        $\|\DD^j F_t g_n(x)\| \to 0$ as $|x| \to \infty$. Hence
        $\|\DD^j F_t g_n\| \in C_\infty(\cX)$, and the uniform limit
        $\|\DD^j F_t f\|$ also belongs to $C_\infty(\cX)$. Therefore
        $F_t f \in D_l$.

        \emph{Step 2 (strong continuity on $C_c^\infty$).}
        Fix $g \in C_c^\infty(\cX)$. By Lemma~\ref{lem:D4-compact-core},
        for every compact $K \subset \cX$ and every $0 \le j \le l$,
        \[
            \sup_{x \in K}\|\DD^j F_t g(x) - \DD^j g(x)\| \to 0
            \quad\text{as }t \downarrow 0.
        \]
        For tail control, fix $\varepsilon > 0$, choose $K = \overline{B_M}$ with
        $M$ large enough that $\|\DD^j g\| < \varepsilon$ outside $K$, and apply
        the Fa\`a di Bruno representation above: each summand of
        $\DD^j F_t g(x)$ is bounded by
        $\|g\|_{C^l} \cdot \E_x\bigl[\mathbf 1_{\{X_t \in \mathrm{supp}\,g\}}
        \cdot \prod_i \|\DD^{|P_i|}X_t(x)\|\bigr]$.
        H\"older's inequality and the escape estimate of
        Lemma~\ref{lem:escape},
        $\sup_{t\in[0,1]}P_x(X_t \in \mathrm{supp}\,g) \to 0$ as $|x|\to\infty$,
        give
        $\sup_{|x|\ge M_1}\|\DD^j F_t g(x)\| < \varepsilon$
        for $M_1$ sufficiently large, uniformly in $t \in [0,1]$. Combining
        the compact and tail estimates yields
        $\|\DD^j(F_t g - g)\|_\infty \to 0$. The base-point term
        $|F_t g(x_*) - g(x_*)|\to 0$ by continuity at $x_*$ and bounded
        convergence. Hence $\|F_t g - g\|_l \to 0$.

        \emph{Step 3 (strong continuity on $D_l$).}
        Let $f \in D_l$ and choose $g_n \in C_c^\infty(\cX)$ with $g_n \to f$
        in $\|\cdot\|_l$. By the triangle inequality,
        \[
            \|F_t f - f\|_l
            \le \|F_t(f - g_n)\|_l + \|F_t g_n - g_n\|_l + \|g_n - f\|_l .
        \]
        Lemma~\ref{lemm:growh_rate} bounds the first summand by
        $C_{l}(1 + e^{\bar\mu_l t})\|f - g_n\|_l$, uniformly for $t \in [0,1]$.
        Step 2 sends the middle summand to zero as $t\downarrow 0$ for each
        fixed $n$. Taking $n$ large first and then $t \downarrow 0$ proves
        $\|F_t f - f\|_l \to 0$.
    \end{proof}

    \begin{lemma} \label{lem:D4-core}
        Let Assumption~\ref{asm:classical} hold with $m\ge4$.  Then
        $(F_t)_{t\ge0}$ is a bounded strongly continuous semigroup on
        $C_{W,\infty}(\cX)$, with
        \[
            \|F_t g\|_W\le e^{\bar\mu_Wt}\|g\|_W.
        \]
        Moreover $F_tD_4\subset D_4$, and $(F_t)_{t\ge0}$ is strongly
        continuous on $D_4$ in the norm $\|\cdot\|_4$.
    \end{lemma}
    \begin{proof}
        The Lyapunov estimate for $W=1+|x|^4$ follows from
        Lemma~\ref{lem:weighted_estim_tensor_poly} with $q=2$, and the
        weighted Feller argument recalled in
        Lemma~\ref{lem:strong_continuity_weighted_generic} gives bounded
        strong continuity on $C_{W,\infty}(\cX)$. The $D_4$-invariance and
        strong continuity in $\|\cdot\|_4$ are
        Lemma~\ref{lem:D4-strong-continuity-core} with $l = 4$.
    \end{proof}

    \begin{lemma}
    \label{lem:D4-domain-core}
        Let Assumption~\ref{asm:classical} hold with $m\ge4$.  For every
        $f\in D_4$,
        \[
            f\in\operatorname{Dom}_{C_{W,\infty}}(L),
            \qquad
            Lf\in\operatorname{Dom}_{C_{W,\infty}}(L),
        \]
        and
        \[
            \|Lf\|_W+\|L^2f\|_W\le C(K,d)\|f\|_4.
        \]
        Consequently,
        \[
            F_t f-f=\int_0^tF_sLf\,ds,
            \qquad
            F_tLf-Lf=\int_0^tF_sL^2f\,ds
        \]
        as identities in $C_{W,\infty}(\cX)$.
    \end{lemma}
    \begin{proof}
        Lemma~\ref{lem:L-weighted} gives
        $Lf,L^2f\in C_{W,\infty}(\cX)$ and the displayed norm estimate.
        To identify the generator in $C_{W,\infty}$, stop the diffusion at
        $\tau_R=\inf\{t:|X_t|\ge R\}$.  For the stopped diffusion, It\^o's
        formula gives
        \[
            \E f(X_{t\wedge\tau_R})-f(x)
            =
            \E\int_0^{t\wedge\tau_R}Lf(X_s)\,ds.
        \]
        The estimate \eqref{eq:L-weighted-estimate}, the Lyapunov bound for
        $W$, and localization let $R\to\infty$ and yield
        \[
            F_t f-f=\int_0^tF_sLf\,ds
            \quad\text{in }C_{W,\infty}(\cX).
        \]
        Hence $f\in\operatorname{Dom}_{C_{W,\infty}}(L)$.  Repeating the same
        argument with $Lf$, whose pointwise generator is $L^2f$, gives
        $Lf\in\operatorname{Dom}_{C_{W,\infty}}(L)$ and
        \[
            F_tLf-Lf=\int_0^tF_sL^2f\,ds.
        \]
        The displayed bound on $\|Lf\|_W+\|L^2f\|_W$ is precisely
        Lemma~\ref{lem:L-weighted}.
    \end{proof}
    
    In the following lemma, we prove an estimate for the continuous semigroup. 

    \begin{lemma} \label{lemm:F_local}
        Let Assumption~\ref{asm:classical} hold. 
        Then for all $0<h \le h_0$,
        \[
        \left\|
           Lf - \frac{F_h f - f}{h}
        \right\|_W
        \;\le\; C\, h\, \|f\|_{4},
        \]
        where $C$ depends only on the constants in
        Assumption~\ref{asm:classical}.
        \end{lemma}
        
        \begin{proof}
        Lemma~\ref{lem:D4-domain-core} gives $f,Lf\in
        \operatorname{Dom}_{C_{W,\infty}}(L)$ and
        $\|L^2f\|_W\le C\|f\|_4$.  We use the Dynkin formula in the Banach
        space $C_{W,\infty}(\cX)$:
        \[
            F_h f - f = \int_0^h L F_s f ds.
        \]
        Applying this formula twice, we obtain 
        \[
        F_h f - f = hLf + \int_0^h\!\!\int_0^s F_r L^2 f \, dr\, ds.
        \]
        Dividing by $h$,
        \[
        \frac{F_h f - f}{h} - Lf
          = \frac{1}{h} \int_0^h\!\!\int_0^s F_r L^2 f\,dr\,ds
          = \int_0^h \Bigl(1-\frac{r}{h}\Bigr) F_r L^2 f\, dr.
        \]
        
        Since $F_r$ is quasi-contractive in $\|\cdot\|_W$, thanks to Lemma~\ref{lem:weighted_estim_tensor_poly},
        \[
        \|F_r g\|_W \le e^{\bar \mu r}\|g\|_W \le e\,\|g\|_W,\qquad 0\le r\le \bar \mu^{-1}.
        \]
        Thus, thanks to \eqref{eq:L2-weighted-estimate}, for all $h \leq h_0 = \bar \mu^{-1}$,
        \[
        \left\|
           \frac{F_h f - f}{h} - Lf
        \right\|_W
        \le \int_0^h e\,\|L^2 f\|_W\, dr = e\,h\,\|L^2 f\|_W \leq e\,h\,l' \, \|f\|_4,
        \]
        which completes the proof.
        \end{proof}

    The next lemma computes the error of a random walk scheme. 
    \begin{lemma} \label{lem:U-local}
        Let Assumption~\ref{asm:classical} hold. Consider a random walk with a one-step transition operator $U_h$ satisfying Assumption~\ref{asm:rw}.
        
        Then there exists a constant $C = C(K,d,M_4) > 0$ such that for all $f\in D_4$ and all $0<h\le 1$,
        \begin{equation}
           \left\|
              \left( \frac{U_h - I}{h} - L \right) f
           \right\|_W
           \;\le\; C\, h\, \|f\|_4,
        \end{equation}
        The constant depends only on $K,d$ and $M_4$.
    \end{lemma}    
    \begin{proof}
    Fix $x\in\cX$ and abbreviate
    \[
       B := b(x), \qquad S_l := \sigma_l(x),\qquad
       \eta := \eta_h(x).
    \]
    We write
    \[
       \eta = hB + \sqrt{h}\,Y, \qquad Y := \sum_{j=1}^d \sigma_j(x) \xi_j.
    \]
    We have $\E[Y]=0$, and by the third--moment assumption
    \[
       \E[Y^{\wedge 3}] = 0.
    \]

    The tensor Taylor expansion with remainder gives
    \[
    \begin{aligned}
       f(x+\eta)
       &= f(x)
        + \nabla f(x)[\eta]
        + \frac12 \DD^2 f(x)[\eta,\eta]
        + \frac{1}{6}\DD^3 f(x)[\eta,\eta,\eta] \\
       &\qquad
        + \frac{1}{24}\DD^4 f(x+\theta\eta)[\eta,\eta,\eta,\eta],
    \end{aligned}
    \]
    for some (random) $\theta\in(0,1)$.
    Let us take expectations of these values from the Taylor extension.
    Since $\E[\xi]=0$ and $\E[\xi_l\xi_m]=\delta_{lm}$,
    \[
       \E[\eta] = hB,
       \qquad
       \E[\eta\wedge\eta]
         = h\sum_{l=1}^d S_l\wedge S_l + h^2 B\wedge B.
    \]
    Hence
    \[
    \begin{aligned}
       \E[f(x+\eta)]
       &= f(x)
          + h \nabla f(x)[B]
          + \frac{h}{2} \sum_{l=1}^d \DD^2 f(x)[S_l,S_l] \\
       &\qquad
          + \frac{h^2}{2}\,\DD^2 f(x)[B,B]
          + R_3(x) + R_4(x),
    \end{aligned}
    \]
    where
    \[
       R_3(x) := \frac{1}{6}\E[\DD^3 f(x)[\eta,\eta,\eta]],\qquad
       R_4(x) := \frac{1}{24}\E[\DD^4 f(x+\theta\eta)[\eta,\eta,\eta,\eta]].
    \]
    The generator is
    \[
       Lf(x) 
       = \nabla f(x)[B] + \frac12\sum_{l=1}^d \DD^2 f(x)[S_l,S_l].
    \]
    Thus
    \[
       (U_h f)(x) - f(x) - hLf(x)
       = \frac{h^2}{2}\DD^2 f(x)[B,B] + R_3(x) + R_4(x).
    \]
    Using $\|B\|\le K(1+\|x\|)$,
    \[
       \left|\frac{h^2}{2}\DD^2 f(x)[B,B]\right|
       \le h^2 \|f\|_4\, K^2(1+\|x\|)^2.
    \]
    Dividing by $hW(x)$ and using
    $(1+\|x\|)^2 \le 2(1+\|x\|^4)=2W(x)$,
    \[
       \frac{1}{W(x)}
       \left|
          \frac{1}{h}\cdot \frac{h^2}{2}\DD^2f(x)[B,B]
       \right|
       \le 2K^2 h\,\|f\|_4.
    \]
    Since $\DD^3 f(x)$ is a symmetric tensor,
    \[
       R_3(x)
       = \frac{1}{6}\DD^3 f(x)\big[\E(\eta^{\wedge 3})\big].
    \]
    We compute $\E(\eta^{\wedge 3})$ explicitly. With $\eta = hB + \sqrt{h}Y$,
    \[
       \eta^{\wedge 3}
       = (hB + \sqrt{h}Y)^{\wedge 3}
       = h^3 B^{\wedge 3}
         + 3h^2\sqrt{h}\, B^{\wedge 2}\wedge Y
         + 3h(\sqrt{h})^2 B\wedge Y^{\wedge 2}
         + (\sqrt{h})^3 Y^{\wedge 3}.
    \]
    Taking expectations and using $\E[Y]=0$ and $\E[Y^{\wedge 3}]=0$,
    \[
       \E(\eta^{\wedge 3})
       = h^3 B^{\wedge 3}
         + 3h^2\, B\wedge \E(Y^{\wedge 2}).
    \]
    Since $\E[\xi_i\xi_j]=\delta_{ij}$,
    $\E(Y^{\wedge 2}) = \sum_{l=1}^d \sigma_l(x)\vee\sigma_l(x)$ is
    deterministic, with norm bounded by
    \[
       \|\E(Y^{\wedge 2})\|
       \le \sum_{l=1}^d \|\sigma_l(x)\|^2
       \le d K^2(1+\|x\|)^2,
    \]
    and $\|B\|\le K(1+\|x\|)$, so
    \[
       \|\E(\eta^{\wedge 3})\|
       \le (1 + 3d) K^3 h^2 (1+\|x\|)^3.
    \]
    Since $(1+\|x\|)^3 \le 8(1+\|x\|^4)=8 W(x)$, we get
    \[
       \|\E(\eta^{\wedge 3})\|
       \le 8 (1 + 3d) K^3 h^2 W(x).
    \]
    Thus,
    \[
       \frac{|R_3(x)|}{hW(x)}
       \le \frac{1}{6}\|\DD^3 f(x)\|\,\|\E(\eta^{\wedge 3})\|
       \le \frac{4}{3} (1 + 3d) K^3  h\,\|f\|_4.
    \]
    The fourth-order term
    \[
       |R_4(x)| \le \frac{1}{24}\|f\|_4\, \E\|\eta\|^4.
    \]
    Using the vector inequality $\|u+v\|^4 \le 8(\|u\|^4 + \|v\|^4)$,
    \[
       \E\|\eta\|^4 \le 8h^4\|B\|^4 + 8h^2\,\E\|Y\|^4.
    \]
    
    For the drift part, we obtain 
    \[
       8h^4\|B\|^4
       \le 8h^4 K^4(1+\|x\|)^4
       \le 64K^4 W(x) h^4.
    \]
    
   Now we proceed with the diffusion part. Thanks to the Cauchy--Schwarz inequality,
   \[
       \left( \sum_{j=1}^d a_j \right)^2 \leq d \sum_{j=1}^d a_j^2.
   \]
    Therefore,
    \[
       \E\|Y\|^4 \le d^3 \sum_{j=1}^d |\sigma_j(x)|^4 |\xi_j|^4  \le 2 d^4 M_4 K^4 (1+|x|^4).
    \]
    Combining the two contributions gives (recall that $h \leq 1$)
    \[
       \E\|\eta\|^4 \le C K^4\,(1 + d^4 M_4)\, W(x)\, h^2,
    \]
    
    Collecting the three contributions, we obtain for all $x\in\cX$ and $0<h\le1$,
    \[
    \begin{aligned}
       \frac{1}{W(x)}
       \left|
          \left(
             \frac{U_h-I}{h}-L
          \right)f(x)
       \right|
       &\le
       \Big(
          2K^2
          + \frac{4}{3}(1 + 3d) K^3
          + C K^4(1+d^4M_4)
       \Big) h\,\|f\|_4 \\
       &\le
       C(K,d,M_4)\, h\,\|f\|_4,
    \end{aligned}
    \]
    which finishes the proof.
    \end{proof}

    The same one-step computation gives the base-space stability of $U_h$ on
    $C_W(\cX)$, which is the second hypothesis of
    Proposition~\ref{prop:1}.

    \begin{lemma} \label{lem:U-Lyapunov-W}
        Let Assumptions~\ref{asm:classical} and~\ref{asm:rw} hold.  Then there
        is a constant $q_W=q_W(K,d,M_4)>0$ such that, for all $0<h\le1$,
        \[
            \E\,W\bigl(x+\eta_h(x)\bigr)\le e^{q_Wh}\,W(x),
            \qquad x\in\cX,
        \]
        and consequently $\|U_hg\|_W\le e^{q_Wh}\|g\|_W$ for every
        $g\in C_W(\cX)$.
    \end{lemma}
    \begin{proof}
        Write $z:=x+hb(x)$ and $Y:=\sum_{\ell=1}^d\sigma_\ell(x)\xi_\ell$, so
        that $x+\eta_h(x)=z+\sqrt h\,Y$ with $\E Y=0$.  Expanding
        $|z+\sqrt h\,Y|^4$ and using $\E Y=0$, the term of order $h^{1/2}$
        vanishes, while the terms of order $h^{3/2}$ and higher are bounded,
        for $0<h\le1$, by $C\,h\,W(x)$, through
        $\E|Y|^2\le dK^2(1+|x|)^2$ and
        $\E|Y|^4\le C(d,M_4)K^4(1+|x|)^4$ together with
        $|z|\le(1+hK)(1+|x|)$.  Hence
        $\E|x+\eta_h(x)|^4\le|x+hb(x)|^4+C\,h\,W(x)$, and the linear growth of
        $b$ gives $|x+hb(x)|^4\le|x|^4+C\,h\,W(x)$.  Therefore
        $\E W(x+\eta_h(x))\le W(x)+C\,h\,W(x)\le e^{q_Wh}W(x)$ with
        $q_W:=C$.  The bound on $\|U_hg\|_W$ follows from
        $|U_hg(x)|\le\|g\|_W\,\E W(x+\eta_h(x))$.
    \end{proof}

    We next record the estimates in the weighted regular spaces
    $\cC^r_{\alpha,\infty}(\cX)$.

    \begin{lemma} \label{lem:weighted-calculus}
        Let $\alpha,\beta\in\R$ and $k\in\N_0$.
        \begin{enumerate}
            \item If $u\in\cC^k_{\alpha,\infty}(\cX)$ and
            $v\in\cC^k_{\beta,\infty}(\cX)$, then
            $uv\in\cC^k_{\alpha+\beta,\infty}(\cX)$ and
            \[
                \|uv\|_{\alpha+\beta,k}
                \le C_{\alpha,\beta,k}\|u\|_{\alpha,k}\|v\|_{\beta,k}.
            \]
            The same estimate holds for contractions of tensor-valued
            functions.
            \item Let
            \[
                \eta_h(x)=h b(x)+\sqrt h\sum_{\ell=1}^d\sigma_\ell(x)\xi_\ell
            \]
            and let Assumptions~\ref{asm:weighted-regular} and
            \ref{asm:rw-bounded} hold.  Then, for $0<h\le h_0$ and
            $0\le\theta\le1$,
            \[
                C^{-1}w_\gamma(x)\le
                w_\gamma(x+\theta\eta_h(x))
                \le Cw_\gamma(x),
                \qquad \gamma\in\R,
            \]
            and, for $1\le j\le k+4$,
            \[
                \|\DD^j\eta_h(x)\|
                \le C\sqrt h\,w_{1-j}(x).
            \]
            Moreover, if $\varphi_h^\theta(x):=x+\theta\eta_h(x)$, then
            \[
                \|\DD\varphi_h^\theta(x)\|\le C,
                \qquad
                \|\DD^j\varphi_h^\theta(x)\|
                \le C\sqrt h\,w_{1-j}(x),
                \quad j\ge2.
            \]
            \item If $g\in\cC^k_{\alpha,\infty}(\cX)$, then
            $x\mapsto \E g(x+\eta_h(x))$ belongs to
            $\cC^k_{\alpha,\infty}(\cX)$.
        \end{enumerate}
    \end{lemma}
    \begin{proof}
        The product estimate follows
        from Leibniz' rule.  For $0\le j\le k$, every term of
        $\DD^j(uv)$ is a contraction of $\DD^\ell u$ with
        $\DD^{j-\ell}v$, and
        \[
            w_{\alpha-\ell}(x)\,w_{\beta-(j-\ell)}(x)
            =w_{\alpha+\beta-j}(x).
        \]
        Dividing by $w_{\alpha+\beta-j}$ and taking the supremum gives the
        estimate.  The same calculation applies to any fixed tensor
        contraction because the contraction norm is bounded by the product of
        the participating tensor norms.  The vanishing-at-infinity condition
        is preserved since each weighted derivative ratio is a product of
        bounded ratios with at least one factor which vanishes at infinity.

        For the comparability estimate, boundedness of $\xi$ and the linear
        growth of the coefficients give
        \[
            |\eta_h(x)|
            \le C\sqrt h\,(1+|x|),
            \qquad 0<h\le1.
        \]
        Choose $h_0$ so that $C\sqrt h_0\le1/2$.  Then
        $1+|x+\theta\eta_h(x)|$ and $1+|x|$ are comparable uniformly in
        $\theta\in[0,1]$ and in the bounded random variable $\xi$.  Raising
        the resulting two-sided estimate to the power $\gamma$ gives the
        displayed weight comparison for positive and negative $\gamma$.
        Differentiating $\eta_h$ gives
        \[
            \DD^j\eta_h(x)
            =
            h\,\DD^j b(x)
            +\sqrt h\sum_{\ell=1}^d
                \DD^j\sigma_\ell(x)\xi_\ell ,
        \]
        and Assumption~\ref{asm:weighted-regular} yields
        $\|\DD^j\eta_h(x)\|\le C\sqrt h\,w_{1-j}(x)$ for $h\le1$.
        The estimates for $\varphi_h^\theta$ follow because
        $\DD\varphi_h^\theta=I+\theta\DD\eta_h$ and
        $\DD^j\varphi_h^\theta=\theta\DD^j\eta_h$ for $j\ge2$.

        Finally put $\varphi_h(x)=x+\eta_h(x)$.  By
        Proposition~\ref{prop:faadibruno}, for $0\le j\le k$,
        $\DD^j(g\circ\varphi_h)(x)$ is a finite sum over partitions
        $P=\{P_1,\dots,P_q\}\in\cP(j,q)$ of contractions
        \[
            \DD^q g(\varphi_h(x))
            \big[
                \DD^{|P_1|}\varphi_h(x),\dots,
                \DD^{|P_q|}\varphi_h(x)
            \big].
        \]
        The $g$-factor is bounded by
        $\|g\|_{\alpha,k}w_{\alpha-q}(\varphi_h(x))$, hence by
        $C\|g\|_{\alpha,k}w_{\alpha-q}(x)$.  Each singleton block contributes
        a bounded first derivative of $\varphi_h$, and each block of size
        $m\ge2$ contributes $C\sqrt h\,w_{1-m}(x)$.  Since
        $\sum_i|P_i|=j$, the product of the weights is
        $w_{\alpha-q}(x)\prod_iw_{1-|P_i|}(x)=w_{\alpha-j}(x)$, with
        singleton factors equal to $w_0=1$.  Thus the weighted derivative
        ratios are uniformly bounded.  Taking expectations preserves the
        same bounds.

        To prove vanishing at infinity, fix $j$ and a partition term.  The
        ratio of the corresponding term to $w_{\alpha-j}(x)$ is bounded by
        a constant times
        \[
            \frac{\|\DD^q g(\varphi_h(x))\|}
                 {w_{\alpha-q}(\varphi_h(x))}
        \]
        multiplied by bounded coefficient ratios.  Since
        $|\varphi_h(x)|\to\infty$ uniformly in the bounded increment variable
        as $|x|\to\infty$, and the displayed ratio for $g$ belongs to
        $C_\infty(\cX)$, every partition term vanishes at infinity.  The
        expectation is over a bounded random variable and therefore preserves
        this limit by dominated convergence.
    \end{proof}

    \begin{lemma}
    \label{lem:weighted-density}
        For every $\alpha\in\R$ and $r\in\N_0$, $C_c^\infty(\cX)$ is dense
        in $\cC^r_{\alpha,\infty}(\cX)$ with respect to
        $\|\cdot\|_{\alpha,r}$.  In particular,
        $\cC^m_{\alpha,\infty}(\cX)$ is dense in
        $\cC^r_{\alpha,\infty}(\cX)$ whenever $m\ge r$.
    \end{lemma}
    \begin{proof}
        Let $f\in\cC^r_{\alpha,\infty}(\cX)$ and choose
        $\chi\in C_c^\infty(\cX)$ with $0\le\chi\le1$, $\chi=1$ on
        $B_1$, and $\chi=0$ outside $B_2$.  Put
        $\chi_R(x)=\chi(x/R)$ and $f_R=\chi_Rf$.  For $0\le j\le r$,
        Leibniz' rule gives
        \[
            \DD^j(f-f_R)
            =
            (1-\chi_R)\DD^j f
            +
            \sum_{\ell=1}^j
            C_{j,\ell}\,\DD^\ell(1-\chi_R)\,\DD^{j-\ell}f .
        \]
        The first term tends to zero in the weighted norm because
        $\DD^j f/w_{\alpha-j}$ vanishes at infinity.  The remaining terms
        are supported in $\{R\le |x|\le2R\}$ and satisfy
        $\|\DD^\ell\chi_R(x)\|\le C_\ell R^{-\ell}\le C_\ell w_{-\ell}(x)$
        there.  Since
        $w_{-\ell}w_{\alpha-(j-\ell)}=w_{\alpha-j}$, these terms also tend
        to zero uniformly after division by $w_{\alpha-j}$.  Hence
        $\|f_R-f\|_{\alpha,r}\to0$.

        For fixed $R$, mollify $f_R$ by a standard mollifier
        $\rho_\varepsilon$.  The functions
        $f_{R,\varepsilon}:=\rho_\varepsilon*f_R$ belong to
        $C_c^\infty(\cX)$ and converge to $f_R$ in the ordinary $C^r$ norm
        on a fixed compact set.  On that compact set the weighted norm is
        equivalent to the usual $C^r$ norm, so
        $\|f_{R,\varepsilon}-f_R\|_{\alpha,r}\to0$.  This proves density.
    \end{proof}

    \begin{lemma} \label{lem:weighted-regular-generator}
        Fix $\alpha\in\R$ and $r\in\N_0$, and assume
        Assumption~\ref{asm:weighted-regular}.  Then there exists
        $C_{\alpha,r}>0$ such that, for every
        $f\in\cC^{r+4}_{\alpha,\infty}(\cX)$,
        \[
            \|Lf\|_{\alpha,r+2}
            +
            \|L^2f\|_{\alpha,r}
            \le
            C_{\alpha,r}\|f\|_{\alpha,r+4}.
        \]
        In particular,
        \[
            L:\cC^{r+2}_{\alpha,\infty}(\cX)
              \to \cC^r_{\alpha,\infty}(\cX)
        \]
        is a bounded operator.
    \end{lemma}
    \begin{proof}
        We write
        \[
            Lf=\DD f[b]
            +\frac12\sum_{\ell=1}^d\DD^2f[\sigma_\ell,\sigma_\ell].
        \]
        Assumption~\ref{asm:weighted-regular} gives
        $\|\DD^j b(x)\|\le Cw_{1-j}(x)$.  By Leibniz' rule, the
        derivatives of the tensor field
        $a:=\sum_{\ell=1}^d\sigma_\ell\otimes\sigma_\ell$ satisfy
        \[
            \|\DD^j a(x)\|\le C w_{2-j}(x),
            \qquad 0\le j\le r+4.
        \]
        For $0\le q\le r+2$, every term in
        $\DD^q(\DD f[b])$ is bounded by a product of the form
        \[
            \|\DD^{q-\ell+1}f(x)\|\,
            \|\DD^\ell b(x)\|
            \le
            C\|f\|_{\alpha,r+4}
            w_{\alpha-(q-\ell+1)}(x)w_{1-\ell}(x)
            =
            C\|f\|_{\alpha,r+4}w_{\alpha-q}(x).
        \]
        The diffusion term is identical, using
        $w_{\alpha-(q-\ell+2)}w_{2-\ell}=w_{\alpha-q}$.  This proves
        $\|Lf\|_{\alpha,r+2}\le C\|f\|_{\alpha,r+4}$.  The same
        derivative expansion proves the vanishing-at-infinity condition:
        after division by $w_{\alpha-q}$ each term contains a weighted
        derivative ratio of $f$, evaluated at $x$, multiplied by bounded
        coefficient ratios.  Thus $Lf\in\cC^{r+2}_{\alpha,\infty}$.
        Applying the already proved bound with $Lf$ in place of $f$ gives
        $L^2f\in\cC^r_{\alpha,\infty}$ and
        $\|L^2f\|_{\alpha,r}\le C\|Lf\|_{\alpha,r+2}\le
        C\|f\|_{\alpha,r+4}$.
    \end{proof}

    \begin{lemma}
        \label{lem:weighted-regular-semigroup-core}
        Fix $\alpha\in\R$ and $r\in\N_0$, and assume
        Assumption~\ref{asm:weighted-regular}.  The diffusion semigroup
        $(F_t)_{t\ge0}$ leaves $\cC^r_{\alpha,\infty}(\cX)$ invariant and is
        strongly continuous on this space.  If $A_{\alpha,r}$ denotes its
        generator on $\cC^r_{\alpha,\infty}(\cX)$, then
        \[
            \cC^{r+2}_{\alpha,\infty}(\cX)
            \subset \operatorname{Dom}(A_{\alpha,r}),
            \qquad
            A_{\alpha,r}f=Lf.
        \]
        The space $\cC^{r+2}_{\alpha,\infty}(\cX)$ is dense in
        $\cC^r_{\alpha,\infty}(\cX)$ and is an invariant core for
        $A_{\alpha,r}$.
        Moreover, for each $T>0$ there is $M_{\alpha,r,T}$ such that
        \[
            \sup_{0\le t\le T}\|F_tg\|_{\alpha,r}
            \le M_{\alpha,r,T}\|g\|_{\alpha,r}.
        \]
    \end{lemma}
    \begin{proof}
        We first record the weighted flow estimates used below.  The
        Lyapunov calculation for the weights $w_\gamma$ gives, for every
        $\gamma\in\R$, $p\ge1$, and $T>0$,
        \[
            \sup_{0\le t\le T}\E w_\gamma(X_t(x))^p
            \le C_{\gamma,p,T}w_\gamma(x)^p .
        \]
        Kunita's equations for the flow jets, combined with
        $\|\DD^q\sigma_\ell(y)\|\le Cw_{1-q}(y)$, yield by induction
        \[
            \sup_{0\le t\le T}\E\|\DD X_t(x)\|^p\le C_{p,T},
            \qquad
            \sup_{0\le t\le T}\E\|\DD^jX_t(x)\|^p
            \le C_{j,p,T}w_{1-j}(x)^p,\quad j\ge2.
        \]
        Indeed, in the $j$th jet equation each forcing term associated with
        a partition $P=\{P_1,\dots,P_q\}$ has weight
        \[
            w_{1-q}(X_s)\prod_{i=1}^q w_{1-|P_i|}(X_s),
        \]
        and after H\"older's inequality, the Lyapunov estimate, and the
        induction hypothesis this becomes $w_{1-j}(x)$ because
        $\sum_i|P_i|=j$.  Gr\"onwall's lemma then gives the displayed bound.

        Let $g\in\cC^r_{\alpha,\infty}(\cX)$.  The derivative
        representation for $F_tg$ follows from Kunita's flow theorem and
        Proposition~\ref{prop:faadibruno}.  For $1\le j\le r$,
        \[
            \DD^jF_tg(x)
            =
            \E\sum_{q=1}^{j}\sum_{P\in\cP(j,q)}
            \DD^q g(X_t(x))
            \big[
                \DD^{|P_1|}X_t(x),\dots,
                \DD^{|P_q|}X_t(x)
            \big],
        \]
        while $F_tg(x)=\E g(X_t(x))$ for $j=0$.  H\"older's inequality and
        the estimates above give
        \[
            \|\DD^jF_tg(x)\|
            \le C_{\alpha,r,T}\|g\|_{\alpha,r}w_{\alpha-j}(x),
            \qquad 0\le j\le r,\quad 0\le t\le T.
        \]
        This proves the boundedness estimate on $\cC^r_\alpha$.

        We next prove the vanishing condition and strong continuity.
        First let $g\in C_c^\infty(\cX)$.  The same derivative
        representation and the continuity of the stochastic flow jets at
        $t=0$ give local uniform convergence
        $\DD^jF_tg\to\DD^jg$ for $0\le j\le r$.  To control the tail, choose
        $\beta<\alpha$.  Since $g$ is compactly supported,
        $g\in\cC^r_{\beta,\infty}(\cX)$, and the estimate already proved
        gives, for $0\le t\le1$,
        \[
            \frac{\|\DD^jF_tg(x)\|}{w_{\alpha-j}(x)}
            \le
            C\|g\|_{\beta,r}w_{\beta-\alpha}(x),
            \qquad |x|\ \text{large}.
        \]
        The right-hand side tends to zero as $|x|\to\infty$, uniformly in
        $t\in[0,1]$.  Thus $F_tg\in\cC^r_{\alpha,\infty}$ and
        $\|F_tg-g\|_{\alpha,r}\to0$ for compactly supported smooth $g$.
        Lemma~\ref{lem:weighted-density} and the boundedness estimate extend
        both statements to arbitrary
        $g\in\cC^r_{\alpha,\infty}(\cX)$.

        If $f\in\cC^{r+2}_{\alpha,\infty}(\cX)$, then
        Lemma~\ref{lem:weighted-regular-generator} gives
        $Lf\in\cC^r_{\alpha,\infty}(\cX)$.  Applying It\^o's formula to the
        stopped diffusion and then removing the stopping by the weighted
        Lyapunov estimates gives
        \[
            F_t f-f=\int_0^tF_sLf\,ds
            \quad\text{in }\cC^r_{\alpha,\infty}(\cX),
        \]
        and therefore
        \[
            \frac{F_tf-f}{t}-Lf
            =
            \frac1t\int_0^t(F_sLf-Lf)\,ds
            \to0
        \]
        in $\cC^r_{\alpha,\infty}$ by strong continuity.  This identifies
        the generator.  Density follows from Lemma~\ref{lem:weighted-density}.
        Invariance of $\cC^{r+2}_{\alpha,\infty}$ follows from the same
        argument applied at derivative order $r+2$.  Hence
        $\cC^{r+2}_{\alpha,\infty}$ is a dense invariant subspace contained
        in $\operatorname{Dom}(A_{\alpha,r})$, and the semigroup
        core criterion makes it an invariant core.
    \end{proof}

    \begin{lemma} \label{lemm:F-local-weighted-regular}
        Fix $\alpha\in\R$ and $r\in\N_0$, and assume
        Assumption~\ref{asm:weighted-regular}.  Then, for every
        $f\in\cC^{r+4}_{\alpha,\infty}(\cX)$ and all sufficiently small
        $h>0$,
        \[
            \left\|
                \frac{F_h-I}{h}f-Lf
            \right\|_{\alpha,r}
            \le
            C_{\alpha,r}h\|f\|_{\alpha,r+4}.
        \]
    \end{lemma}
    \begin{proof}
        By Lemma~\ref{lem:weighted-regular-generator},
        $Lf\in\cC^{r+2}_{\alpha,\infty}(\cX)$ and
        $L^2f\in\cC^{r}_{\alpha,\infty}(\cX)$.
        Lemma~\ref{lem:weighted-regular-semigroup-core} therefore gives
        $f\in\operatorname{Dom}(A_{\alpha,r})$ and
        $Lf\in\operatorname{Dom}(A_{\alpha,r})$, where $A_{\alpha,r}$ is the
        generator of $(F_t)$ on $\cC^r_{\alpha,\infty}(\cX)$.  Dynkin's
        formula in this Banach space, applied twice, gives
        \[
            \frac{F_hf-f}{h}-Lf
            =
            \int_0^h\left(1-\frac{s}{h}\right)F_sL^2f\,ds.
        \]
        The regularity estimate from Section~\ref{sec:sensitivities} gives
        boundedness of $F_s$ on $\cC^r_{\alpha,\infty}(\cX)$ for
        $0\le s\le1$.  Combining this bound with
        Lemma~\ref{lem:weighted-regular-generator} yields the claim.
    \end{proof}

    \begin{lemma}
        \label{lem:U-local-weighted-regular}
        Fix $\alpha\in\R$ and $r\in\N_0$.  Let Assumptions
        \ref{asm:weighted-regular}, \ref{asm:rw}, \ref{asm:rw-bounded}, and
        \ref{asm:rw-weighted-stability} hold.  Then there are constants
        $C_{\alpha,r},q_{\alpha,r}>0$ and $h_0>0$ such that, for all
        $0<h\le h_0$,
        \[
            \|U_hg\|_{\alpha,r}
            \le
            e^{q_{\alpha,r}h}\|g\|_{\alpha,r},
            \qquad
            g\in\cC^r_{\alpha,\infty}(\cX),
        \]
        and
        \[
            \left\|
                \left(\frac{U_h-I}{h}-L\right)f
            \right\|_{\alpha,r}
            \le
            C_{\alpha,r}h\|f\|_{\alpha,r+4},
            \qquad
            f\in\cC^{r+4}_{\alpha,\infty}(\cX).
        \]
    \end{lemma}
    \begin{proof}
        Put
        \[
            \eta_h(x):=h b(x)+\sqrt h\sum_{\ell=1}^d\sigma_\ell(x)\xi_\ell.
        \]
        Lemma~\ref{lem:weighted-calculus} gives the weight comparability,
        the derivative bounds for $\eta_h$, and the preservation of
        $\cC^r_{\alpha,\infty}$ by $U_h$.  The displayed stability estimate
        is precisely Assumption~\ref{asm:rw-weighted-stability}.  It remains
        to prove the local consistency estimate.

        For consistency, use Taylor's formula with integral fourth-order
        remainder:
        \[
        \begin{aligned}
            f(x+\eta_h)
            &=f(x)+\DD f(x)[\eta_h]
            +\frac12\DD^2f(x)[\eta_h,\eta_h]
            +\frac16\DD^3f(x)[\eta_h,\eta_h,\eta_h]\\
            &\quad
            +\frac1{6}\int_0^1(1-\theta)^3
            \DD^4f(x+\theta\eta_h)
            [\eta_h,\eta_h,\eta_h,\eta_h]\,d\theta .
        \end{aligned}
        \]
        The coefficient $1/6$ is the integral form of the fourth-order
        remainder, since
        $\int_0^1(1-\theta)^3\,d\theta=1/4$.
        Let
        \[
            a(x):=\sum_{\ell=1}^d\sigma_\ell(x)\otimes\sigma_\ell(x).
        \]
        The moment assumptions give
        \[
            \E\eta_h=hb,\qquad
            \E(\eta_h^{\otimes2})=h\,a+h^2 b^{\otimes2},
        \]
        and, because all third moments of $\xi$ vanish,
        \[
            \E(\eta_h^{\otimes3})
            =
            h^2\,\mathcal S(b\otimes a)+h^3 b^{\otimes3},
        \]
        where $\mathcal S(b\otimes a)$ denotes the finite symmetrised sum of
        the three tensors obtained by placing $b$ in one of the slots.  Hence
        \[
        \begin{aligned}
            (U_h-I-hL)f(x)
            &=
            \frac{h^2}{2}\DD^2f(x)[b,b]
            +\frac{h^2}{6}\DD^3f(x)[\mathcal S(b\otimes a)]\\
            &\quad
            +\frac{h^3}{6}\DD^3f(x)[b,b,b]
            +R_{4,h}f(x),
        \end{aligned}
        \]
        where
        \[
            R_{4,h}f(x)
            :=
            \frac16\int_0^1(1-\theta)^3
            \E\,\DD^4f(x+\theta\eta_h)
            [\eta_h,\eta_h,\eta_h,\eta_h]\,d\theta .
        \]

        Differentiate the four displayed terms.  For the first term,
        $\DD^j\{\DD^2f[b,b]\}$ is a finite sum of contractions of
        $\DD^{j_0+2}f$ with derivatives $\DD^{j_1}b$ and $\DD^{j_2}b$,
        where $j_0+j_1+j_2=j$.  After division by $w_{\alpha-j}$, each
        product is bounded by
        \[
            \|f\|_{\alpha,r+4}
            \frac{
                w_{\alpha-(j_0+2)}w_{1-j_1}w_{1-j_2}
            }{w_{\alpha-j}}
            =
            \|f\|_{\alpha,r+4}.
        \]
        The same calculation applies to
        $\DD^3f[\mathcal S(b\otimes a)]$ because
        $\|\DD^\ell a(x)\|\le Cw_{2-\ell}(x)$, and to
        $\DD^3f[b,b,b]$ because the coefficient weight is
        $w_{1-j_1}w_{1-j_2}w_{1-j_3}$.  Thus these three deterministic
        terms contribute at most
        $C h^2\|f\|_{\alpha,r+4}w_{\alpha-j}(x)$ for $0\le j\le r$.

        For the integral remainder, apply Proposition~\ref{prop:faadibruno}
        and Leibniz' rule to
        \[
            \DD^4f(x+\theta\eta_h)
            [\eta_h,\eta_h,\eta_h,\eta_h].
        \]
        Every $j$th derivative is a finite sum of contractions containing one
        derivative $\DD^{4+q}f(x+\theta\eta_h(x))$, with $q\le j$, several
        derivatives of $\varphi_h^\theta(x)=x+\theta\eta_h(x)$, and four
        factors obtained by differentiating the four copies of $\eta_h$.
        Lemma~\ref{lem:weighted-calculus} gives the weight comparison
        $w_{\alpha-(4+q)}(x+\theta\eta_h(x))\le
        Cw_{\alpha-(4+q)}(x)$ and the bounds
        $\|\DD^\ell\eta_h(x)\|\le C\sqrt h\,w_{1-\ell}(x)$.
        Since four copies of $\eta_h$ are present before differentiation,
        every term contains the factor $h^2$ and its coefficient weights
        multiply exactly to $w_{\alpha-j}(x)$.  Boundedness of $\xi$ permits
        taking expectation inside the same estimate.  Therefore
        \[
            \|\DD^jR_{4,h}f(x)\|
            \le C_{\alpha,r}h^2
            \|f\|_{\alpha,r+4}w_{\alpha-j}(x).
        \]
        The same product estimates also show that each weighted derivative
        ratio vanishes at infinity, because one of the factors is a weighted
        derivative ratio of $f$ evaluated at $x$ or at
        $x+\theta\eta_h(x)$ and the latter tends to infinity uniformly in
        the bounded increment variable.  Combining the deterministic and
        remainder estimates gives
        \[
            \|\DD^j[(U_h-I-hL)f](x)\|
            \le
            C_{\alpha,r}h^2\|f\|_{\alpha,r+4}w_{\alpha-j}(x),
            \qquad 0\le j\le r.
        \]
        Dividing by $h$ and taking the maximum over $j$ proves the local
        error estimate.
    \end{proof}

    {\bf Completion of the proof of Theorem \ref{prop:uniform}}.
    Apply Proposition~\ref{prop:1} with
    \[
        B = C_{W,\infty}(\cX), \qquad D = D_4.
    \]
    \begin{enumerate}
        \item The quasi-contraction property $\|F_t f\|_W \leq e^{\bar \mu t} \|f\|_W$ and strong continuity on $B$ follow from Lemma~\ref{lem:D4-core}.
        \item The random-walk operator satisfies
        $\|U_hg\|_W\le e^{q_Wh}\|g\|_W$ for $g\in C_W(\cX)$ and $0<h\le1$, by
        Lemma~\ref{lem:U-Lyapunov-W}.
        \item Corollary~\ref{coro:norm_estim} and Proposition~\ref{prop:renorming} give an equivalent norm $\|\cdot\|_4^*$ on $D_4$ such that $\|\cdot\|_4\le\|\cdot\|_4^*\le C_4\|\cdot\|_4$ and $\|F_t f\|_4^*\le e^{mt}\|f\|_4^*$.  Since $D_4\hookrightarrow C_{W,\infty}(\cX)$ continuously, we multiply this equivalent norm by a fixed constant, if necessary, so that $\|f\|_W\le\|f\|_4^*$.
        \item The inequality 
        \[
            \left\| \frac{F_h - I}{h}f - Lf \right\|_W \leq \chi_h \|f\|_4 \leq \chi_h \|f\|^*_4
        \]
        follows from Lemma~\ref{lemm:F_local}.
        \item  The inequality 
        \[
            \left\| \frac{U_h - I}{h}f - Lf \right\|_W \leq \epsilon_h \|f\|_4 \leq \epsilon_h \|f\|_4^*
        \]
        follows from Lemma~\ref{lem:U-local}.
    \end{enumerate}
    Proposition~\ref{prop:1} gives the estimate in the equivalent norm, and $\|f\|_4^*\le C_4\|f\|_4$ returns it to the stated $D_4$ norm.

    \subsection{Proof of Theorem \ref{theo:uniform-weighted-regular}}
    We apply Proposition~\ref{prop:1} with
    \[
        B=\cC^r_{\alpha,\infty}(\cX),
        \qquad
        D=\cC^{r+4}_{\alpha,\infty}(\cX).
    \]
    The boundedness and strong continuity of $F_t$ on $B$ and $D$ follow
    from Lemma~\ref{lem:weighted-regular-semigroup-core}; after applying
    Proposition~\ref{prop:renorming} to the $D$-norm, we may write
    $\|F_t f\|_D\le e^{\mu_{\alpha,r}t}\|f\|_D$.  The stability of $U_h$ on
    $B$ and the local error estimate for $U_h$ are contained in
    Lemma~\ref{lem:U-local-weighted-regular}; the local error estimate for
    $F_h$ is Lemma~\ref{lemm:F-local-weighted-regular}.  Proposition
    \ref{prop:1} then gives \eqref{eq:uniform-weighted-regular}; applying the
    killed part of the same proposition gives
    \eqref{eq:uniform-weighted-regular-killed}.

    \subsection{Proof of Theorem \ref{theo:frac-abstract}}
    Define
    \[
        \hat F_t := e^{-ct}F_t,
        \qquad
        \hat U_h := e^{-ch}U_h,
    \]
    and decompose
    \[
        \left\|
            \E \hat F_{\hat\sigma_T} f - \E \hat U_h^{n_T^h} f
        \right\|_B
        \le I + II,
    \]
    where
    \[
        I :=
        \left\|
            \E \hat F_{N_T^h} f - \E \hat U_h^{n_T^h} f
        \right\|_B,
        \qquad
        II :=
        \left\|
            \E \hat F_{\hat\sigma_T} f - \E \hat F_{N_T^h} f
        \right\|_B.
    \]
    Here the norm is the norm of the abstract space $B$.  We start with
    $I$.  For every deterministic $s\ge0$, the assumed killed deterministic
    estimate yields
    \[
        \|\hat F_s f - \hat U_h^{\lfloor s/h\rfloor}f\|_B
        \le C_{\mathrm{RW}} h e^{(\mu_D-c)s}\|f\|_D
    \]
    If $c\ge\mu_D$, this gives immediately
    \[
        I
        \le
            \E\bigl[\|\hat F_{N_T^h}f - \hat U_h^{n_T^h}f\|_B\bigr]
        \le
        C_{\mathrm{RW}} h \|f\|_D.
    \]
    If $\bar\mu_B\le c<\mu_D$, put
    \[
        z_h:=\frac{\log(1/h)}{2(\mu_D-c)}.
    \]
    For $h$ sufficiently small we have $z_h>t_0$.  Splitting the
    expectation according to $\{N_T^h\le z_h\}$ and $\{N_T^h>z_h\}$ gives
    \[
    \begin{aligned}
        I
        &\le
        C_{\mathrm{RW}}h e^{(\mu_D-c)z_h}\|f\|_D
        +2C_B\|f\|_D\,P(N_T^h>z_h)  \\
        &\le
        C h^{1/2}\|f\|_D
        +2C_B\|f\|_D\,P(\Phi^h_{z_h}\le T).
    \end{aligned}
    \]
    By Proposition~\ref{prop:frac2}, enlarged in the constant to cover
    $T\le1$ and the replacement of $<$ by $\le$, we have
    \[
        P(\Phi^h_{z_h}\le T)\le C_T z_h^{-1/\beta}.
    \]
    Hence
    \[
        I
        \le
        C_T(\log(1/h))^{-1/\beta}\|f\|_D,
    \]
    after absorbing the stronger term $h^{1/2}$ into the logarithmic bound
    for small $h$.

    Now consider $II$. Writing $\mu_{\hat\sigma_T}$ and $\mu_{N_T^h}$ for the laws of $\hat\sigma_T$ and $N_T^h=h n_T^h$, we have
    \[
        II
        =
        \left\|
            \int_{[0,\infty)} \hat F_s f \, d(\mu_{\hat\sigma_T}-\mu_{N_T^h})(s)
        \right\|_B.
    \]
    Since $f\in D$, the map $s\mapsto \hat F_s f$ is continuously
    differentiable in $B$ and
    \[
        \frac{d}{ds}\hat F_s f
        =
        \hat F_s (L-c)f.
    \]
    Hence, using $c\ge\bar\mu_B$ and the assumed derivative estimate,
    \[
        \left\|\frac{d}{ds}\hat F_s f\right\|_B
        \le
        C_L e^{(\bar\mu_B-c)s}\|f\|_D
        \le C_L\|f\|_D,
        \qquad s\ge0.
    \]
    Write
    $G_h(s):=P(\hat\sigma_T>s)-P(N_T^h>s)
    =(\mu_{\hat\sigma_T}-\mu_{N_T^h})\bigl((s,\infty)\bigr)$, so that
    $d(\mu_{\hat\sigma_T}-\mu_{N_T^h})(s)=-\,dG_h(s)$ on $(0,\infty)$.  The Banach-valued
    Lebesgue--Stieltjes integration-by-parts identity for the bounded
    variation function $s\mapsto G_h(s)$ on $[0,\infty)$ then reads
    \[
        \int_{[0,\infty)}\hat F_s f\,d(\mu_{\hat\sigma_T}-\mu_{N_T^h})(s)
        =
        -\bigl[\hat F_s f\cdot G_h(s)\bigr]_0^\infty
        +\int_0^\infty\Bigl(\tfrac{d}{ds}\hat F_s f\Bigr)\,G_h(s)\,ds.
    \]
    At $s=0$, both inverse clocks satisfy $\hat\sigma_T,N_T^h>0$ almost
    surely, so $G_h(0)=0$.  As $s\to\infty$, the tail estimate
    Proposition~\ref{prop:frac2} and the analogous tail bound for
    $\hat\sigma_T$ (a polynomial of order $-1/\beta$ in $s$) give
    $P(\hat\sigma_T>s),P(N_T^h>s)\to0$, hence $G_h(s)\to0$; combined with the
    uniform bound $\|\hat F_sf\|_B\le C_B\|f\|_D$ from the standing
    hypothesis, the boundary term at infinity vanishes.  Therefore
    \[
        II
        \le
        C_L \|f\|_D
        \int_0^\infty
        \bigl|
            P(\hat\sigma_T > s) - P(N_T^h > s)
        \bigr|\,ds.
    \]
    The continuous identity is
    $\{\hat\sigma_T>s\}=\{\hat\Sigma_s\le T\}$.  For the discrete clock,
    since $n_T^h=\min\{n:S_n^h>T\}$,
    \[
        \{N_T^h>s\}
        =
        \{h n_T^h>s\}
        =
        \{n_T^h>\lfloor s/h\rfloor\}
        =
        \{S_{\lfloor s/h\rfloor}^h\le T\}
        =
        \{\Phi_s^h\le T\}.
    \]
    These identities also underlie Lemma~\ref{lem:inverse-clock-rate}, which
    converts the forward subordinator estimate of Theorem~\ref{theo:frac1}
    into the inverse-clock bound.  Applying that lemma,
    \[
        II
        \le
        C_L C_{\mathrm{clock}}(T) h^{\chi(\beta)} \|f\|_D.
    \]

    Combining the bounds for $I$ and $II$ gives, in the case
    $c\ge\mu_D$,
    \[
        \|u_h(T,\cdot)-u(T,\cdot)\|_B
        \le
        C_{\mathrm{frac}} \bigl(h+h^{\chi(\beta)}\bigr)\|f\|_D,
    \]
    which is \eqref{eq:frac-abstract}.  In the case
    $\bar\mu_B\le c<\mu_D$, the same clock bound together with the
    logarithmic estimate for $I$ gives \eqref{eq:frac-abstract-log}.

    \subsection{Proof of Corollaries \ref{cor:frac-bounded-derivatives} and
    \ref{cor:frac-weighted-regular}}
    For Corollary~\ref{cor:frac-bounded-derivatives}, apply
    Theorem~\ref{theo:frac-abstract} with
    \[
        B=C_{W,\infty}(\cX),\qquad D=D_4.
    \]
    Strong continuity of $(F_t)_{t\ge0}$ on $B$ and the inclusion
    $D\subset\mathrm{Dom}(L_B)$ with $L_B|_D=L$ are supplied by
    Lemmas~\ref{lem:D4-core} and~\ref{lem:D4-domain-core}, which also
    give $F_s f-f=\int_0^s F_rLf\,dr$ in $B$; combining with the
    semigroup identity yields $s\mapsto\hat F_s f\in C^1([0,\infty);B)$
    with $\tfrac{d}{ds}\hat F_s f=\hat F_s(L-c)f$.
    The deterministic killed estimate is the killed part of
    Theorem~\ref{prop:uniform}.  The base-space killed growth bound for $F_s$
    follows from the weighted Lyapunov estimate for $W$
    (Lemma~\ref{lem:weighted_estim_tensor_poly}), and the one for
    $U_h^{\lfloor s/h\rfloor}$ from the one-step bound
    Lemma~\ref{lem:U-Lyapunov-W}.
    The bound on
    $e^{-cs}F_s(L-c)f$ follows from Lemma~\ref{lem:L-weighted} and the
    $C_{W,\infty}$-growth estimate for $F_s$.

    For Corollary~\ref{cor:frac-weighted-regular}, apply
    Theorem~\ref{theo:frac-abstract} with
    \[
        B=\cC^r_{\alpha,\infty}(\cX),
        \qquad
        D=\cC^{r+4}_{\alpha,\infty}(\cX).
    \]
    Strong continuity of $(F_t)_{t\ge0}$ on $B$, the inclusion
    $D\subset\mathrm{Dom}(A_{\alpha,r})$ with $A_{\alpha,r}|_D=L$, and
    the resulting $C^1$-differentiability of $s\mapsto\hat F_s f$ in $B$
    are all contained in Lemma~\ref{lem:weighted-regular-semigroup-core}.
    The deterministic killed estimate is
    \eqref{eq:uniform-weighted-regular-killed}.  The base-space killed
    growth bound is supplied by Lemma~\ref{lem:weighted-regular-semigroup-core}
    for $F_s$ and by Assumption~\ref{asm:rw-weighted-stability} for
    $U_h^{\lfloor s/h\rfloor}$.  The bound on
    $e^{-cs}F_s(L-c)f$ follows from
    Lemma~\ref{lem:weighted-regular-generator} and the weighted regularity
    estimate for $F_t$ on $\cC^r_{\alpha,\infty}(\cX)$.

\section*{Declarations}

%Some journals require declarations to be submitted in a standardised format. Please check the Instructions for Authors of the journal to which you are submitting to see if you need to complete this section. If yes, your manuscript must contain the following sections under the heading `Declarations':

This project was supported by Vega Institute Foundation.

%\begin{itemize}
%\item No funding was received to assist with the preparation of this manuscript.
%\item The authors have no relevant financial or non-financial interests to disclose.
%\item Ethics approval and consent to participate: not applicable.
%\item Consent for publication: not applicable.
%\item Data availability: not applicable.
%\item Materials availability: not applicable.
%\item Code availability: not applicable.
%\item Author contribution: not applicable.
%\end{itemize}

%\noindent
%If any of the sections are not relevant to your manuscript, please include the heading and write `Not applicable' for that section. 


\ifnum\UseBiblatex=1
\printbibliography
\else
\bibliography{sn-bibliography}% common bib file
\fi

\end{document}